\documentclass{amsart}
\usepackage{pgf,pgfarrows,pgfnodes,pgfautomata,pgfheaps,pgfshade,hyperref, amssymb,enumerate,amsmath}
\usepackage[all]{xy}
\usepackage[capitalize]{cleveref}
\usepackage{mathtools}
\usepackage[shortlabels]{enumitem}
\usepackage{stmaryrd}
\usepackage{tcolorbox}
\usepackage{bm}

\newtheorem{theorem}{Theorem}[section]
\newtheorem*{theorem*}{theorem}
\newtheorem{proposition}[theorem]{Proposition}
\newtheorem{lemma}[theorem]{Lemma}
\newtheorem{corollary}[theorem]{Corollary}

\newtheorem{remark}[theorem]{Remark}

\newtheorem{question}[theorem]{Question}

\usepackage{euscript,epsfig}
\usepackage{tikz}
\usetikzlibrary{graphs}
\usetikzlibrary[graphs]
\usetikzlibrary{arrows}
\usepackage{tikz-cd}
\usetikzlibrary{decorations.markings}
\usepackage[colorinlistoftodos]{todonotes}

\def\Z{{\mathbb Z}}

\def\N{{\mathbb N}}

\def\cP{{\mathcal P}}

\def\IT{{\rm IT}}
\def\IN{{\rm IN}}
\def\Per{{\rm Per}}
\def\Aper{{\rm Aper}}
\def\Ker{{\rm Ker}}
\def\Stab{{\rm Stab}}

\begin{document}

\title[Sequence entropy and independence]{Sequence entropy and independence in free and minimal actions} 
\author{Jaime Gómez}
\author{Irma León-Torres}
\author{Víctor Muñoz-López}

\address{J. G\'omez, Mathematical Institute, University of Leiden, The Netherlands}
\email{j.a.gomez.ortiz@math.leidenuniv.nl}  
\thanks{J. Gómez thanks the Physics Institute of the Universidad Autónoma de San Luis Potosí, where part of this work was done, for its hospitality.}

\address{I. Le\'on-Torres, Physics Institute, Universidad Aut\'onoma de San Luis Potos\'i, M\'exico.}
\email{yooirma@gmail.com}

\address{V. Muñoz-L\'opez, Physics Institute, Universidad Aut\'onoma de San Luis Potos\'i, M\'exico.}
\email{soygenn@gmail.com}
\thanks{I. León-Torres and V. Muñoz-López were supported by a CONAHCyT fellowship for their PhD studies and by grant K/NCN/000198 of the NCN}

\keywords{}


\begin{abstract}

For every countable infinite group that admits $\mathbb{Z}$ as a homomorphic image, we show that for each $m\in\N$, there exists a minimal action whose topological sequence entropy is $\log(m)$.
Furthermore, for every countable infinite group $G$ that contains a finite index normal subgroup $G'$  isomorphic to $\Z^r$, for every $m\in \N$, we found a free minimal action
with topological sequence entropy $\log(n)$, where $m\leq n\leq m^{2^r[G:G']}$.
If, instead, $G$ is assumed to be an infinite countable group that is a direct sum of finite groups, the previous result holds with the topological sequence entropy being equal to $\log(m)$.
In each of the above cases, we also show that the aforementioned  minimal actions  admit non-trivial independence tuples of size $n$ but do not admit non-trivial independence tuples of size $n+1$ for some $n\geq m$.

\end{abstract}

\maketitle
\section{Introduction}

Topological entropy is a measure of the complexity or unpredictability of a dynamical system.
It serves as an invariant that helps to identify when two dynamical systems are not conjugate.
Sequence entropy, a finer notion of entropy, was introduced by Kushnirenko (\cite{Ku67}) for measure preserving systems and by Goodman (\cite{Go74}) for (topological) dynamical systems and is helpful to distinguish systems of zero entropy.
The sequence entropy is similar to the classical entropy, the difference is that in the refinement for the sequence entropy is made using sequences of the acting group. 
It was noticed in \cite{HuYe09}, that the topological sequence entropy of a system can only take values of the form $\log(n)$, for $n\in\mathbb{N}\cup\{\infty\}$.


\medskip

According to classical topological entropy, a question arise:
Given a non-negative real number, can we construct a minimal dynamical system, over an amenable countable group, with the condition that the entropy of this system is the real number previously considered? 
Downarowicz, in \cite{Do91}, has answered this question in an affirmative way for $\mathbb{Z}$-actions; this result was extended for amenable group actions in \cite{Kr07, LaSt19}.
To answer this question, a particular family of dynamical systems,  known as \textit{Toeplitz subshifts}, was used.
The previous question can be refined by imposing additional conditions on the dynamical system, such as freeness or unique ergodicity, see \cite{Dr24} and \cite{IwLa91} for a treatment of this question for the uniquely ergodic condition in $\mathbb{Z}$ and $\mathbb{Z}^d$ actions.

\medskip

In this paper, we are interested in constructing minimal dynamical systems that have zero topological entropy but can achieve arbitrarily large, yet finite, topological sequence entropy.
Moreover, for these examples we can guarantee that there exists a natural number for which the $\IT$-tuples of the system are trivial. 
These results are related to the recent article \cite{LiWaXu25}, as one of their main results states that for every minimal dynamical  system $(X,G)$, with $G$ abelian, there exists $K\in \mathbb{N}\cup\{\infty\}$ with the topological sequence entropy of the system greater or equal than $\log(K)$ such that the system is a regular $K$-to-one extension of its maximal equicontinuous factor $(X_{\rm eq}, G)$.
That is, the set $\{x\in X_{\rm eq}: |\pi^{-1}(x)|=K\}$ has full measure under the unique invariant measure of $(X_{\rm eq},G)$, where $\pi:X\to X_{\rm eq}$ is the factor map associated to the maximal equicontinuous factor. 
Therefore, our results can be read as a realization problem for the topological sequence entropy.

The first result of this document reads as follows.

\begin{theorem}\label{main-theo-3}
    Let $G$ be a countable group such that $\mathbb{Z}$ is a homomorphic image of $G$.
    For every $m\geq 1$, there exists a minimal subshift $X\subseteq \{0,\dots,m-1\}^G$ with topological sequence entropy equal to $\log(m)$.
\end{theorem}

In the previous theorem, if $G=\mathbb{Z}$, then the minimal subshift is also free, i.e., every element in this subshift has trivial stabilizer. 
However, when $G\neq \mathbb{Z}$, this is no longer true. 
When we regard about free dynamical systems for the question related to topological sequence entropy, we obtain the following.

\begin{theorem}\label{theo-main}
    Let $G$ be an infinite countable group that contains a finite index normal subgroup $G'$ with $G'\cong \mathbb{Z}^r$, for some $r\in\mathbb{N}$.
For every $m\geq 2$, there exists a free minimal subshift $(X,\sigma,G)$ with zero entropy such that its topological sequence entropy is a value between $\log(m)$ and $\log(m^{2^r[G:G']})$.
\end{theorem}

If $G$ is considered to be an infinite countable group that is a direct sum of finite groups, the previous result holds with the topological sequence entropy being equal to $\log(m)$ (see Proposition \ref{prop: sum-finite}).

\medskip

To construct these examples we use Toeplitz subshifts.
These dynamical systems were defined by Jacobs and Keane in \cite{JaKe69} for $\mathbb{Z}$-actions and \cite{Co06}, \cite{CoPe08} for countable group actions. 
This family of dynamical systems has been useful by their ergodic diversity and entropy, as we can observe in the works  \cite{CoPe14}, \cite{Do91},  \cite{FuKw19}, \cite{Kr07}, \cite{Wi84} to mention some of them.
Moreover, the maximal equicontinuous factor of these systems are well-known, they are called \textit{$G$-odometers}. 
This relation between Toeplitz subshifts and $G$-odometers allow us to study the sequence entropy using independence tuples.

\medskip

Independence is a combinatorial tool, introduced by Kerr and Li in \cite{KeLi07}, that helps to study different dynamical properties, such as entropy, sequence entropy, and tameness.
See the survey \cite{GaLi24} for an  overview of recent results related to independence, and see \cite{Hu06} and \cite{HuLiShYe03} for works in which related phenomena were studied for Abelian groups.
We are interested in IN-tuples and IT-tuples, these tuples characterize topological sequence entropy and tameness, respectively.
In case of IN-tuples, positive topological sequence entropy is characterized with the existence of IN-pairs (IN-tuples of size two) of different elements. 
In a more general way, since IN-tuples are just the same of sequence entropy pairs, in case of tuples with different elements (see \cite{KeLi07}), we can assert from \cite{HuYe09} that topological sequence entropy is $\log(n)$ where $n$ is the natural number such that the dynamical system has an IN-tuple of size $n$ with different elements and all $\IN$-tuples of size $n+1$ having at least two elements that are equal.

\medskip

On the other hand, a dynamical system is called \emph{tame} if the cardinality of its Ellis semigroup is at most that of the continuum.
Moreover, IT-tuples characterize tameness as follows: A dynamical system is tame if and only if all IT-pairs are those pairs of points which have same elements. 
In \cite{HuLi21}, it was studied the relationship between IT-tuples and a bound of the number of ergodic measures of minimal dynamical systems.
In more recent works \cite{LiWaXu25} and \cite{LiXuZh25}, this bound was improved, and in the latter an optimal upper bound was established.

We say that a dynamical system is \emph{$m$-tame} if all IT-tuples of size $m$ have at least two coordinates with the same element. 
A system which is not $m$-tame, for some $m\in\mathbb{N}$, we call it \emph{$m$-untame}.
From the proof of Theorem \ref{main-theo-3} and Theorem \ref{theo-main}, we also deduce that the minimal subshifts presented in these theorems are $n$-untame, but $n+1$-tame, for $n=m$ in Theorem \ref{main-theo-3}, and for some $m\leq n\leq m^{2^r[G:G']}$ in Theorem \ref{theo-main}.


\medskip


\medskip

This work is divided as follows.
In Section \ref{Sec-2}, we give basic notions concerning topological dynamics, as well as some background with respect to Toeplitz  subshifts, $G$-odometers, independence and sequential entropy. 
At the end of Section \ref{Sec-2}, we discuss Theorem \ref{theo-main} for $\mathbb{Z}$-actions.
In Section \ref{sec-3}, we present the proof of Theorem \ref{main-theo-3}.
The proof of Theorem \ref{theo-main} is provided in Section \ref{sec-4}, we split the cases in two, when the group is abelian finitely generated with positive free rank and when the group contains a normal subgroup which is isomorphic to $\mathbb{Z}^r$ for some $r>0$.

\medskip

\textbf{Acknowledgments:} The authors would like to thank Felipe García-Ramos for his guidance, insightful comments and valuable support during the development of this work.
Additionally, the authors are grateful with Tobias Jäger for meaningful comments.

\section{Preliminaries}\label{Sec-2}
In this article, integers, nonnegative integers and natural numbers are denoted by $\Z,\Z_+$ and $\N$, respectively. 

\subsection{Dynamical systems and invariant measures}
Throughout this paper $G$ represents an infinite countable group with identity $1_G$, and $X$ a compact metric space.
By a \emph{dynamical system} $(X,\varphi,G)$ we refer to a continuous action $\varphi: G\times X\to X$.
We denote the image $\varphi(g,x)$ as $\varphi^g(x)$.
A system $(X,\varphi,G)$ is called \emph{minimal} if for every $x\in X$ its orbit $O_\varphi(x)=\{\varphi^g x: g\in G\}$ is dense in $X$ and it is called \emph{free} if $\mbox{Stab}_\varphi(x)=\{1_G\}$ for every $x\in X$, where $\mbox{Stab}_\varphi(x)=\{g\in G: \varphi^g(x)= x\}$ for $x\in X$.
A Borel subset $A\subseteq X$ is called invariant if $\varphi^g(A)= A$ for every $g\in G$.

\medskip

An \emph{invariant measure} of the dynamical system $(X,\varphi,G)$ is a Borel probability measure $\mu$ of $X$ that verifies $\mu(\varphi ^g (A))=\mu(A)$, for every $g\in G$ and every Borel set $A$.
An invariant measure $\mu$ is called \emph{ergodic} if $\mu(A)=0$ or $\mu(A)=1$ for every invariant subset $A\subset X$.
The set of invariant measures of $(X,\varphi,G)$ is denoted by $M_G(X)$.
If $\mu$ is an invariant measure of $(X,\varphi, G)$, the quadruple $(X,\varphi,G,\mu)$ is called a {\em probability-measure-preserving (p.m.p) dynamical system}.
Two p.m.p dynamical systems $(X,\varphi, G, \mu)$ and $(Y,\psi, G,\nu)$ are {\em measure conjugate} if: (i) there exist conull sets $X'\subseteq X$ and $Y'\subseteq Y$ satisfying $\varphi^g(X')\subseteq X'$ and $\psi^g(Y')\subseteq Y'$ for all $g\in G$, and (ii) there exists a bijective map $f:X'\to Y'$ such that $f, f^{-1}$ are both measurable, $\nu(A)=\mu(f^{-1}(A))$ for every measurable set $A\subseteq Y'$, and $f(\varphi^g(x))=\psi^g(f(x))$ for all $x\in X'$ and $g\in G$. In this case, we say that $f$ is a {\em measure conjugacy}.

\medskip

Let $(X,\varphi, G)$ and $(Y,\psi,G)$ two dynamical systems. We say that $(Y,\psi,G)$ is \emph{factor} of $(X,\varphi, G)$  through $\pi$, if $\pi:X\to Y$ is a surjective continuous map such that $\pi(\varphi^gx)=\psi^g(\pi(x))$ for each $g\in G$ and $x\in X$.
In this case, we say that $\pi$ is a \emph{factor map}.
A factor map $\pi:X\to Y$ is called \emph{almost} one-to-one if the set of points in $Y$ having only one preimage is residual.
In the case that $X$ is minimal, this is equivalent to the existence of an element in $Y$ with only one preimage.

\medskip

A dynamical system $(X,\varphi,G)$ is called \emph{equicontinuous} if the family of maps $\{\varphi^g:X\to X\}_{g\in G}$ is equicontinuous. 
There exists a unique factor $(X_{\rm eq},\varphi_{\rm eq},G)$ of $(X,\varphi,G)$ such that $(X_{\rm eq},\varphi_{\rm eq},G)$ is equicontinuous, and if $(Y,\psi,G)$ is an equicontinuous factor of $(X,\varphi,G)$, then $(Y,\psi,G)$ is a factor of $(X_{\rm eq},\varphi_{\rm eq},G)$.
The system $(X_{\rm eq},\varphi_{\rm eq},G)$ is called the \textit{maximal equicontinuous factor} and we denote by $\pi_{\rm eq}$ its corresponding factor map.

\subsection{Toeplitz $G$-subshifts and $G$-odometers}\label{G-Toeplitz} 
Let $\Sigma$ be a finite set with at least two elements endowed with the discrete topology.
Consider the set
$$\Sigma^G=\{(x(g))_{g\in G}: x(g)\in \Sigma \mbox{ for every } g\in G\}$$
endowed with the product topology.
The \emph{(left) shift action} $\sigma$ of $G$ on $\Sigma^G$
is defined as $\sigma^g(x)(h)=x(g^{-1}h)$, for all $g,h\in G$ and $x\in \Sigma^G$.
The dynamical system $(\Sigma^G,\sigma,G)$ is known as a \textit{full $G$-shift} (or full shift), and each closed subset $X\subseteq \Sigma^G$ that is a $\sigma$-invariant set is called a \textit{subshift}. 
We also call to the system $(X,\sigma, G)$ as a \emph{subshift}.

An element $x\in \Sigma^G$ is called a \textit{Toeplitz array}
or a \textit{Toeplitz element} if for every $g\in G$
there exists a finite index subgroup $\Gamma$ of $G$ such that 
$\sigma^\gamma(x)(g)=x(\gamma^{-1} g)=x(g)$, for every $\gamma\in \Gamma$.
Let  $\Gamma$ be a  subgroup of $G$, $x\in \Sigma^G$
and $\alpha \in\Sigma$.
We define 
$$\Per(x,\Gamma,\alpha)=\{g\in G: x(\gamma^{-1}g)=\alpha \mbox{ for each }\gamma\in\Gamma\}$$
and
$$\Per(x,\Gamma)=\bigcup_{\alpha \in\Sigma} \Per(x,\Gamma,\alpha).$$
The following lemma can be found in \cite{CoPe08}.
\begin{lemma}\label{no-normal}
    Let $\Gamma$ be a subgroup of $G$ and $x\in \Sigma^G$.
    For each $g\in G$ and $\alpha\in \Sigma$, it holds that 
    \begin{align*}
        \Per(\sigma^g x,\Gamma,\alpha)=g\Per(x,g^{-1}\Gamma g,\alpha).
    \end{align*}
\end{lemma}

A subgroup $\Gamma\leq G$  is a \textit{group of periods of} $x$
if $\Per(x,\Gamma)\neq\emptyset$.
A group of periods $\Gamma$ is called an \textit{essential group of periods} of $x$
if for every $g\in G$ that satisfies  $\Per(x,\Gamma,\alpha)\subseteq \Per(\sigma^gx,\Gamma,\alpha)$
for every $\alpha\in\Sigma$, we have $g\in\Gamma$.

Let $x\in\Sigma^G$ be a Toeplitz array. 
A \emph{period structure} of $x$ is a nested sequence of finite index subgroups $(\Gamma_n)_{n\in\N}$ of $G$ satisfying $G=\bigcup_{n\in\N}\Per(x,\Gamma_n)$ and $\Gamma_n$ is an essential group of periods of $x$, $n\in\N$.
A \textit{Toeplitz $G$-subshift} (or Toeplitz subshift) is the subshift generated by the closure of the $\sigma$-orbit of a Toeplitz array.

Suppose $(\Gamma_n)_{n\in\N}$ is a nested sequence of finite index subgroups of $G$ such that $\bigcap_{n=1}^\infty\Gamma_n=\{1_G\}$.
The \textit{$G$-odometer} associated to $(\Gamma_n)_{n\in\N}$
is defined as
$$\overleftarrow{G}=\left\{(g_n\Gamma_n)_{n\in\N}\in\prod_{n=1}^{\infty}G/\Gamma_n:\varphi_n(g_{n+1}\Gamma_{n+1})=g_n\Gamma_n, \mbox{ for every }n\in\N\right\},$$
where $\varphi_n:G/\Gamma_{n+1}\to G/ \Gamma_n$ is the canonical projection
for every $n\in\N$.
There is a natural action of $G$ on $\overleftarrow{G}$ given by the left coordinate-wise multiplication, i.e., $\phi^g((g_n\Gamma_n)_{n\in\mathbb{N}})=(gg_n\Gamma_n)_{n\in\mathbb{N}}$.
The dynamical system $(\overleftarrow{G},\phi,G)$ is a free equicontinuous minimal Cantor system and it is also known as a $G$-\emph{odometer}.
If $x\in\Sigma^G$ is a Toeplitz array with period structure $(\Gamma_n)_{n\in\N}$, then $\overleftarrow{G}$ is the maximal equicontinuous factor of the Toeplitz subshift $X=\overline{O_\sigma(x)}$ given by the factor map $\pi_{\rm eq}:X\to \overleftarrow{G}$, defined as $\pi_{\rm eq}(x)=(g_n\Gamma_n)_{n\in\N}$ when $\sigma^{g_n} x\in C_n$, for each $n\in\N$,
where 
$$C_n=\{y\in X:\Per(y,\Gamma_n,\alpha)=\Per(x,\Gamma_n,\alpha)\mbox{, for every }\alpha\in\Sigma\}.$$ 
\begin{lemma}\label{Per-lemma}
    Let $x\in\Sigma^G$ be a Toeplitz array with period structure given by $(\Gamma_n)_{n\in\mathbb{N}}$ and $\overline{O_\sigma(x)}\subseteq \Sigma^G$ its Toeplitz subshift associated.
    Let $\overleftarrow{G}$ be the $G$-odometer associated to $(\Gamma_n)_{n\in\mathbb{N}}$ and  $\pi:\overline{O_\sigma(x)}\to\overleftarrow{G}$ be the factor map from $\overline{O_\sigma(x)}$ into $\overleftarrow{G}$.
    If $\pi(y)=\pi(y')$ for $y,y'\in \overline{O_\sigma(x)}$, then $y(g)=y'(g)$ for each $g\in \bigcup_{n\in\mathbb{N}}\Per(y,g_n^{-1}\Gamma_ng_n)=\bigcup_{n\in\mathbb{N}}\Per(y',g_n^{-1}\Gamma_ng_n)$.
\end{lemma}
\begin{proof}
Suppose $\pi(y)=\pi(y')=(g_n\Gamma_n)_{n\in\mathbb{N}}$ and $g_n\in G$, $n\in\mathbb{N}$.
By definition, we have $\sigma^{g_n}y,\sigma^{g_n}y'\in C_n $ for every $n\in\mathbb{N}$.
Therefore, $\Per(\sigma^{g_n}y,\Gamma_n,\alpha)=\Per(\sigma^{g_n}y',\Gamma_n,\alpha)=\Per(x,\Gamma_n,\alpha)$ for each $\alpha\in\Sigma$ and $n\in\mathbb{N}$. 
Using Lemma \ref{no-normal}  we deduce that $\Per(y,g_n^{-1}\Gamma_ng_n,\alpha)=\Per(y',g_n^{-1}\Gamma_ng_n,\alpha)=g_n^{-1}\Per(x,\Gamma_n,\alpha)$ and we conclude the lemma.
\end{proof}

\subsection{Independence and sequence entropy}
Let $(X,\varphi,G)$ be a dynamical system, $\mathcal{U}$ an open cover of $X$ and $S=(g_n)_{n\in\N}$ a sequence of $G$.
The \textit{sequence entropy of $(X,\varphi,G)$
with respect to $\mathcal{U}$ along} $S$ is defined as
$$ h_{top}^*(X,\varphi,G,\mathcal{U};S)=\limsup_{n\to\infty} \frac{1}{n} \log N \left(\bigvee_{i=1}^{n} \varphi^{g^{-1}_n} \mathcal{U}\right), $$
where $N(\mathcal{C})$ is the minimal cardinality among all cardinalities of subcovers of $\mathcal{C}$.
The \textit{topological sequence entropy} of $(X,\varphi,G)$ is given by
\begin{align*}
h_{top}^*(X,\varphi,G)=&\sup_{\mathcal{U}, S} h_{top}^*(X,\varphi,G,\mathcal{U};S),
\end{align*}
where the supremum is taken over all open covers of $X$ and all sequences of $G$. 
The system $(X,\varphi,G)$ is called \emph{null} if $h^*_{top}(X,\varphi,G)=0$ and \emph{non-null} otherwise.
For a tuple $A=(A_1,\ldots,A_k)$ of subsets of $X$, we say that a set $J\subseteq G$ is an \textit{independence set} for $A$ if for every non-empty finite subset $I\subseteq J$ and function $s:I\to \{1,\ldots,k\}$ we have $$\bigcap_{g\in I}\varphi^{g^{-1}}A_{s(g)}\neq \emptyset.$$

A tuple $x=(x_1,\ldots,x_n)\in X^n$ is called an $n$-\textit{$\IN$-tuple} (\textit{$n$-$\IT$-tuple}) if for any product neighborhood $U_1\times\cdots\times U_n$ of $x$, the tuple $(U_1,\ldots,U_n)$ has arbitrarily large finite independence sets (has an infinite independence set).
We denote the set of $n$-$\IN$-tuples and $n$-$\IT$-tuples by $\IN_n(X)$ and $\IT_n(X)$, respectively. Note that $\IT_n(X)\subseteq \IN_n(X)$, for each $n\in\N$.

\begin{remark}\label{equalIT}
    If $(X,\varphi, G)$ is a minimal dynamical system, then the tuple $(x,\dots, x)\in X^n$ belongs to $\IT_n(X)$ for each $x\in X$. 
    Furthermore, if $(x_1, x_2,\dots, x_n)\in \IN_n(X)$, then $(x_1,x_1,x_2,\dots,x_n)\in\IN_{n+1}(X)$.
\end{remark}

For $n\in \N$, define the set
\begin{align*}\label{def:triangle}
 \triangle^{({n})}(X)=\{(x_1,\dots,x_n)\in X^n:x_i=x_j \mbox{ for some } i,j\in \{1,\dots,n\}, i\neq j\}. 
\end{align*}

\begin{proposition}[{\cite[Proposition 5.4]{KeLi07}}]\label{closedtuples}
    Let $(X,\varphi,G)$ be a dynamical system. The following are true:
    \begin{enumerate}
        \item Let $(A_1,\ldots,A_k)$ be a tuple of closed subsets of $X$ which has arbitrarily large finite independence sets. Then there exists an $\IN$-tuple, $(x_1,\ldots,x_k)$ with $x_j\in A_j$ for all $1\leq j \leq k$.
        \item $\IN_2(X)\setminus \triangle^{(2)}(X)\neq \emptyset$ if and only if $(X,\varphi,G)$ is non-null.
        \item $\IN_k(X)$ is a closed $G$-invariant subset of $X^k$.
    \end{enumerate}
\end{proposition}

\begin{proposition}[{\cite[Theorem 4.4]{HuYe09}}]\label{htop*}
Let $(X,\varphi,G)$ be a dynamical system. 
Then $h^*_{top}(X,\varphi,G)=\log(\max\{n:\IN_n(X)\setminus\triangle^{(n)}(X)\neq\emptyset\})$.
\end{proposition}

A tuple $(x_1,\ldots,x_n)\in X^n$ is called \textit{$n$-regionally proximal} if for each $\varepsilon>0$, there exist $x'_1,\ldots,x'_n\in X$ such that $d(x_i,x'_i)<\varepsilon$ for all $i\in\{1,\ldots,n\}$, and there exists $g\in G$ such that $d(\varphi^ gx'_i,\varphi^ gx'_j)<\varepsilon$ for all $i,j\in \{1,\ldots,n\}$. We denote by $Q_n(X,G)$ the collection of all $n$-regionally proximal tuples of $(X,\varphi,G)$.

\medskip

The following proposition relates the $\IN$-tuples to the regionally proximal tuples of a system. 

\begin{proposition}\label{It-to-Reg}
    Let $(X,\varphi,G)$ be a dynamical system, $n\in\mathbb{N}$ and $x=(x_1,\ldots,x_n)\in X^n$. 
    If $x\in \IN_n(X)$, then $x$ is a $n$-regionally proximal tuple. 
    In particular, every $n$-$\IT$-tuple is an $n$-regionally proximal tuple.
\end{proposition}

\begin{proof}
    Let $x=(x_1,\dots, x_n)$ be an $n$-$\IN$-tuple. 
    Let $\varepsilon>0$ and $U_i=B(x_i,\frac{\varepsilon}{2})$.
    For the sets $U_1,\ldots, U_n$ there exists an independence set $J\subseteq G$ with $|J|\geq 2$.
    Let $g,h\in J$ with $g\neq h$.
    By independence, for $i\in\{1,\ldots,n\}$ and $\sigma_i:\{g,h\}\to \{1,\ldots,n\}$
    given by $\sigma_i(g)=i$ and $\sigma_i(h)=1$, we have that
\[\varphi ^{g^{-1}}U_i\cap \varphi^{h^{-1}}U_1\neq\emptyset.\]
    Then, there exists $y_i\in X$ such that $\varphi^{g} y_i\in U_i$ and $\varphi^{h} y_i\in U_1$.
    Let $x_i'=\varphi^{g}y_i\in U_i$ and $g_0:=hg^{-1}$.  
    Observe 
    $\varphi^{g_0}x_i'=\varphi^{hg^{-1}}\varphi^gy_i=\varphi^hy_i\in U_1$ for each $i\in\{1\ldots, n\}$.
    Thus, $d(\varphi^{g_0}x_i',\varphi^{g_0}x_j')<\varepsilon$ for every $i,j\in\{1,\ldots, n\}$.
    Since $x_i'\in U_i$ for every $i\in \{1,\dots, n\}$, we deduce that  $x$ is an $n$-regionally proximal tuple. 
\end{proof}

\begin{proposition}[{\cite[Proposition 3.7]{HuLi21}}]\label{medidas-IT}
Let $(X,\varphi,G)$ be a minimal dynamical system and assume that the maximal equicontinuous factor map is almost one-to-one.
If $|M^{\rm erg}(X,G)|> l-1$, for some $l\geq 2$, then $\IT_l(X)\setminus\triangle^{(l)}(X)\neq \emptyset$.
In particular, it implies that $\IN_l(X)\setminus \triangle^{(l)}(X)\neq \emptyset.$
\end{proposition}

In the meantime, the previous proposition was improved in \cite[Corollary 4.2]{LiXuZh25}. In particular, the assumption that the maximal equicontinuous factor map is almost one-to-one can be weakened to being almost finite-to-one.

\subsection{$\mathbb{Z}$-Toeplitz subshifts} \label{ZToeplitz}
In this subsection, we observe that  the Toeplitz subshifts constructed in \cite[Section 3]{Wi84} satisfy a stronger version of Theorem \ref{theo-main} when $G=\mathbb{Z}$. 

\medskip

Right below, we recall the construction of the $\Z$-Toeplitz sequences made in \cite{Wi84}: 
Let $m\in \N$ and $\Sigma=\{0,1,\ldots,m-1\}$. 
Fix a sequence of  positive integers $(p_n)_{n\in\N}$ such that
$p_i|p_{i+1}$ and $p_i\geq 3$, $\frac{p_{i+1}}{p_i}\geq 3$,
for all $i\in\N$.
Consider the sequence $(\alpha_i)_{i\in\N}\subseteq \mathbb{N}$, given by 
$\alpha_i=j\in\{0,1,\dots, m-1\}$ whenever $j\equiv i \pmod{m}$.
The Toeplitz element $\eta\in\Sigma^\Z$ is defined inductively.

\medskip

{\bf Step 1:}
Set $\eta(n)=\alpha_1$, for every $n\equiv 0$ or $n\equiv-1 \mod p_1$.

\medskip

{\bf Step 2:}
For each $k\in\Z$ consider $J(1,k)=[kp_1+1,(k+1)p_1-1)$.
Define $\eta(n)=\alpha_2$ for every $n\in J(1,k)$ with $k\equiv 0$ or $k\equiv -1\mod \frac{p_2}{p_1}$.

\medskip

{\bf Step i+1:}
For $i\in\N$,  denote by $J(i,k)$ the set of elements $n\in [kp_i,(k+1)p_i)$
for which $\eta(n)$ has not been defined at the end of the $i^{\mathrm{th}}$
step.
Define $\eta(n)=\alpha_{i+1}$ for $n\in J(i,k)$ with $k\equiv 0$ or $k\equiv-1 \pmod{ \frac{p_{i+1}}{p_i}}$.

\medskip

It is also guaranteed that if the series $\sum_{i=1}^\infty \frac{p_{i}}{p_{i+1}}$ converges, then the set of ergodic measures of the Toeplitz subshift $\overline{O_\sigma(\eta)}\subseteq \Sigma^\mathbb{Z}$ 
is in a one-to-one correspondence with $\Sigma$. 
Furthermore, it is shown that these systems have zero entropy and the odometer $\overleftarrow{G}$ associated to the sequence $(p_n)_{n\in\N}$ is the maximal equicontinuous factor of $\overline{O_\sigma(\eta)}$ with factor map $\pi_{\rm eq}:\overline{O_\sigma(\eta)}\to \overleftarrow{G}$.

\medskip

Let us recall the following lemma, which is important for the proof of Proposition \ref{theo-main-2}.

\begin{lemma}[{\cite[Lemma 3.3]{Wi84}}]\label{Aper-lemma}
For every $g\in \overleftarrow{G}$ and $\alpha\in \Sigma$
there exists $x\in \pi_{\rm eq}^{-1}(g)$ with $x(n)=\alpha$,
for all $n\in \mathbb{Z}\setminus \bigcup_{m\in\mathbb{N}}\Per(x,\Gamma_m)$. 
Conversely, for every $x\in\pi_{\rm eq}^{-1}(g)$ there exists $\alpha\in \Sigma$ such that $x(n)=\alpha$ for every $n\in\mathbb{Z}\setminus \bigcup_{m\in\mathbb{N}}\Per(x,\Gamma_m)$.
\end{lemma}

The following proposition indicates that the previous systems are the one satisfying Theorem \ref{theo-main} for $G=\mathbb{Z}$.

\begin{proposition}\label{theo-main-2}
    For every $m>1$, there exists a free minimal subshift $X$ on $m$ symbols with zero entropy such that $h^*_{top}(X,\sigma,\mathbb{Z})=\log(m)$, $\IT_m(X)\setminus\triangle^{(m)}(X)\neq\emptyset$ and $\IT_{m+1}(X)\setminus \triangle^{(m+1)}(X)=\emptyset$.
\end{proposition}
\begin{proof}
Let $m\geq 2$, $\Sigma=\{0,1,\dots, m-1\}$ and $\eta\in \Sigma^\mathbb{Z}$ the Toeplitz element defined above with periodic structure given by $(p_i)_{i\in\mathbb{N}}$ such that $\sum_{i=1}^\infty \frac{p_{i}}{p_{i+1}}$ converges. 
Let $\overline{O_\sigma(\eta)}\subseteq \Sigma^\mathbb{Z}$ be the Toeplitz subshift associated to $\eta$ and $\overleftarrow{G}$ the $\mathbb{Z}$-odometer associated to the sequence $(p_i)_{i\in\mathbb{N}}$.
In the light of Lemma \ref{Per-lemma}, if $x,x'\in \overline{O_\sigma(\eta)}$ are distinct elements such that $\pi(x)=\pi(x')$, then $x(g)= x'(g)$ for $g\in  \bigcup_{n\in\mathbb{N}}\Per(x,\Gamma_n)=\bigcup_{n\in\mathbb{N}}\Per(x',\Gamma_n)$.
Therefore, Lemma \ref{Aper-lemma} guarantees that the map $\pi$ is at most   $m$-to-one.
Therefore,  we conclude that $Q_{m+1}(\overline{O_\sigma(\eta)},\mathbb{Z})\setminus \triangle^{(m+1)}(\overline{O_\sigma(\eta)})=\emptyset$. (See for instance \cite[Chapter 9]{Au89}).
Applying Proposition \ref{It-to-Reg} we guarantee that $\IN_{m+1}(\overline{O_\sigma(\eta)})\setminus\triangle^{(m+1)}(\overline{O_\sigma(\eta)})=\emptyset$.

On the other side, since this system has exactly $m$ ergodic measures, Proposition \ref{medidas-IT} guarantees that $\IT_{m}(\overline{O_\sigma(\eta)})\setminus \triangle^{(m)}(\overline{O_\sigma(\eta)})\neq \emptyset$.
The Proposition \ref{htop*} concludes the proof. 
\end{proof}

\section{Minimal actions}\label{sec-3}
In this section we present the proof of Theorem \ref{main-theo-3}. 
To complete that, we mention some facts related to cellular automata induced by a group homomorphism.
See \cite{CSVZ23} for a detailed presentation of this topic.

\medskip

Let $G_1$, $G_2$ be countable infinite groups, and $\Sigma$ a finite set. 
For each $\phi: G_1\to G_2$ group homomorphism define the map $\phi^*:\Sigma^{G_2}\to \Sigma^{G_1}$  by $\phi^*(x)=x\circ \phi$.
In this section, we denote by $\sigma_{G_i}$ the shift action of $G_i$  on $\Sigma^{G_i}$, $i\in\{1,2\}$, defined in Section \ref{G-Toeplitz}.

\begin{proposition}[{\cite[Lemma 1 and Lemma 2]{CSVZ23}}]
   Let $\phi: G_1\to G_2$ be a group homomorphism. 
   The cellular automaton $\phi^*:\Sigma^{G_2}\to \Sigma^{G_1}$ is continuous and $\phi$-equivariant.
   That is, for every $g\in G_1$ and $x\in \Sigma^{G_2}$ we have
    \begin{align*}
        \sigma_{G_1}^g\phi^*(x)=\phi^*(\sigma_{G_2}^{\phi(g)}x).
    \end{align*}
    Furthermore, if $\phi$ is surjective, then $\phi^*$ is injective.
\end{proposition}

Next, we relate these notions with minimal subshifts.
\begin{proposition}\label{minimal} 
Let $\phi: G_1\to G_2$ be a surjective group homomorphism.
    If $X\subseteq \Sigma^{G_2}$ is a minimal subshift, then $\phi^*(X)\subseteq \Sigma^{G_1}$ is a minimal subshift.
\end{proposition}
\begin{proof}
    The previous proposition guarantees that $\phi^*(X)$ is a subshift in $\Sigma^{G_1}$.
   Suppose that $X$ is a minimal subshift.
    Let $y,y'\in\phi^*(X)$.
There exist $x,x'\in X$ such that $y=\phi^*(x)$ and $y'=\phi^*(x')$.
    We aim to prove that $y'\in\overline{O_{\sigma_{G_1}}(y)}$.
    Using that $X$ is minimal, there exists a sequence $(h_i)_{i\in\mathbb{N}}\subseteq G_2$ such that $\sigma_{G_2}^{h_i}x\to x'$ as $i\to\infty$. 
    Since $\phi^*$ is continuous, we obtain that $\phi^*(\sigma_{G_2}^{h_i}x)\to \phi^*(x')$ as $i\to\infty$.
    Now, the fact that $\phi$ is surjective implies the existence of $g_i\in G_1$ satisfying $\phi(g_i)=h_i$ for $i\in\mathbb{N}$.
     Finally, the $\phi$-equivariance of $\phi^*$ implies that $\sigma_{G_1}^{g_i}\phi^*(x)\to \phi^*(x')$ when $i\to\infty$ as required.
\end{proof}
In the following proposition, we describe the relation between  $\phi$-equivariant maps and $\IN$-tuples of a minimal subshift. 
This gives a generalization of \cite[Proposition 5.4 (4)]{KeLi07}.

\begin{proposition}\label{IN-tuples+phi}
Let $\phi:G_1\to G_2$ be a surjective group homomorphism, and let $X\subseteq\Sigma^{G_2}$ and $Y\subseteq\Sigma^{G_1}$ be minimal subshifts.
If $\pi:X\to Y$ is a surjective continuous map that is $\phi$-equivariant, then $(\pi\times\dots\times \pi)(\IN_n(X))=\IN_n(Y)$
for each $n\in\mathbb{N}$.
\end{proposition}

\begin{proof}
Let $(y_1,\ldots,y_n)\in (\pi\times \dots \times \pi)(\IN_n(X))$ 
and $V_1\times\cdots\times V_n$ be a product neighborhood of $(y_1,\ldots,y_n)$.
Then, there exists $(x_1,\ldots,x_n)\in \IN_n(X)$ such that
$y_i=\pi(x_i)$, for all $i\in\{1,\ldots,n\}$.
Consider $U_i:=\pi^{-1}(V_i)$, for every $i\in\{1,\ldots,n\}$.
If $J'\subseteq G_2$ is an independence set for $(U_1,\ldots, U_n)$,
then for each $h\in J'$ consider $g_h\in G_1$ such that $\phi(g_h)=h$.

We claim that  $J=\{g_h\in G_1: h\in J'\}\subseteq G_1$ is an independence set for $(V_1,\dots, V_n)$.
Indeed, let $s:J\to\{1,\dots,n\}$ be an arbitrary function and 
define $s':J'\to\{1,\dots, n\}$ as $s'(h):=s(g_h)$, $h\in J'$.
Using that $J'\subseteq G_2$ is an independence set for $(U_1,\dots, U_n)$, there exists $x\in \bigcap_{h\in J'}\sigma_{G_2}^{h^{-1}} U_{s'(h)}$.
Consequently, $\pi(\sigma_{G_2}^{h}x)=\sigma_{G_1}^{g_h}\pi(x)\in V_{s(g_h)}$ for each $g_h\in J$.
Therefore, $$\pi(x)\in\bigcap_{g_h\in J} \sigma_{G_1}^{g_h^{-1}}U_{s(g_h)},$$ which implies that $J$ is an independence set for $(V_1,\dots, V_n)$.
Furthermore, we have that $|J|=|J'|$.
Since $(x_1,\ldots,x_n)\in \IN_n(X)$, the tuple $(U_1,\ldots,U_n)$ has arbitrarily large finite independence sets.
Hence, $(y_1,\dots, y_n)\in \IN_n(Y)$.

\medskip

Conversely, let $(y_1,\ldots,y_n)\in \IN_n(Y)$. 
Assume that $y_1\neq y_i$ for each $i\in\{2,\dots, n\}$.
Let $\overline{B_k(y)}$ denote the closed ball in $Y$ of radius $\frac{1}{k}$ and center in $y\in Y$, i.e., $\overline{B_k(y)}=\overline{B(y,\frac{1}{k})}$, $k\in\mathbb{N}$.
Since $Y$ is a metric space, there exists $K\in\mathbb{N}$ such that $\overline{B_k(y_1)}\cap \overline{B_k(y_i)}=\emptyset$ for every $i\in\{2,\dots,n\}$ and $k\geq K$. 
As before, for $k\geq K$, we can guarantee that if $J$ is an independence set for $(\overline{B_k(y_1)},\dots, \overline{B_k(y_n)})$, then $\phi(J)$ is an independence set for $(\pi^{-1}(\overline{B_k(y_1)}),\dots, \pi^{-1}(\overline{B_k(y_n)}))$.

Now, we  prove that $|\phi(J)|=|J|$.
Let $g,h\in J$ be such that $g\neq h$.
Consider $s_0:J\to\{1,\ldots,n\}$ given by $s_0(g)=1$ and $s_0(h)=2$.
Then, there exists $y\in \pi(X)$ such that $\sigma_{G_1}^gy\in \overline{B_k(y_1)}$ and $\sigma_{G_1}^hy\in \overline{B_k(y_2)}$.
Using that $\pi$ is surjective, there exists $x\in X$ such that $\pi(x)=y$ and by the $\phi$-equivariance of $\pi$ we conclude
$\pi(\sigma_{G_2}^{\phi(g)}x)=\sigma_{G_1}^g y$ and $\pi(\sigma_{G_2}^{\phi(h)}x)=\sigma_{G_1}^h y$.
Thus,  $\sigma_{G_2}^{\phi(g)}x\in  \pi^{-1}(\overline{B_k(y_1)})$ and $\sigma_{G_2}^{\phi(h)}x\in  \pi^{-1}(\overline{B_k(y_2)})$, which is only possible if $\phi(g)\neq \phi(h)$ since these sets are disjoint.

For each $k\geq K$, Proposition \ref{closedtuples} guarantees that there exists an $n$-$\IN$ tuple $(x_k^1,\ldots x_k^n)$ in $ \pi^{-1}(\overline{B_k(y_1)})\times \cdots\times\pi^{-1}(\overline{B_k(y_n)})$. 
Since $\IN_n(X)$ is closed, then, possibly after taking a subsequence, $\lim_{k\to\infty} x^i_k=z_i$ for every $i\in \{1,\ldots,n\}$ and $(z_1,\dots, z_n)\in\IN_n(X)$.
On the other hand, the definition of the tuples $(x_k^1,\dots, x_k^n)$ guarantees $\lim_{k\to\infty}\pi(x_k^i)=y_i$, $1\leq i\leq n$.
Therefore, we obtain $y_i=\pi(z_i)$, $1\leq i\leq n$, and conclude the proof assuming that the elements in the $n$-tuple are pairwise different.
Otherwise, we use Remark \ref{equalIT} to restrict to the previous case.
\end{proof}
\begin{remark}
The conclusion of the previous proposition holds even if we replace $\IN$-tuples with $\IT$-tuples.
\end{remark}

\begin{proof}[Proof of Theorem \ref{main-theo-3}] Let $G$ be a group such that $\mathbb{Z}$ is a homomorphic image of $G$ with the group homomorphism given by $\phi: G\to \mathbb{Z}$.
Let $m\in\mathbb{N}$ and let $X\subseteq \{0,\dots, m-1\}^\mathbb{Z}$ be the Toeplitz subshift presented in subsection \ref{ZToeplitz}.
From Proposition \ref{minimal}, we know that $\phi^*(X)$ is a minimal subshift of $\{0,\dots, m-1\}^G$ such that $\Ker(\phi)$ is contained in the stabilizer of each element $y\in\phi^*(X)$, that is, $\Ker(\phi)\subseteq \bigcap_{y\in\phi^*(X)}\Stab(y)$.
Therefore, using the fact that $\phi^*$ is injective along with Propositions \ref{theo-main-2} and \ref{IN-tuples+phi} we conclude that $\IN_m(\phi^*(X))\setminus \triangle^{(m)}(\phi^*(X))\neq \emptyset$ and $\IN_{m+1}(X)\setminus\triangle^{(m+1)}(X)=\emptyset$.
The same conclusion holds if we replace $\IN$ with $\IT$.
Hence, we conclude the proof using Proposition \ref{htop*}.
\end{proof}

\section{Free minimal actions}\label{sec-4}

Recall that every infinite minimal $\mathbb{Z}$-dynamical system is free. 
However, when the acting group is different from $\mathbb{Z}$, this is no longer true. 
Consequently, the notion of freeness acquires more relevance (see, for instance, \cite{Ga20}, \cite{HjMo06}).
Moreover, this notion has proven useful in characterizing the residual finiteness of a group through the existence of some free Toeplitz subshift (see \cite{Kr07}).

\medskip

In this section, we provide a proof of Theorem \ref{theo-main}. 
That is, we describe the zero entropy minimal free actions with bounded positive sequence entropy.
Firstly, we focus on finitely generated abelian groups, and afterward, we provide a proof for groups that have a normal subgroup isomorphic to $\mathbb{Z}^d$ for some $d\in\mathbb{N}$. 
To this end, we begin by recalling the construction of the Toeplitz arrays made in \cite{CeCoGo23}.
\subsection{Toeplitz arrays}\label{Sec-3}
Let $G$ be a countable infinite group, 
 $\Sigma=\{1,2,\dots, n\}$ for  $n\geq 2$,  and $(\Gamma_i)_{i\in\mathbb{N}}$ a strictly decreasing sequence of normal subgroups of finite index of $G$ such that $\bigcap_{i\in\mathbb{N}} \Gamma_i=\{1_G\}$.
Lemma 2.9 in \cite{CeCoGo23} guarantees the existence of an increasing sequence of finite sets $(D_i)_{i\in\mathbb{N}}$  of $G$ such that for each $i\geq 1$,
\begin{itemize}
    \item $D_i$ is a fundamental domain of $G/\Gamma_i$, that is, $D_i$ contains exactly one representative from each coset of $G/\Gamma_i$.
    \item $\{1_G\}\subseteq D_i\subseteq D_{i+1}$.
    \item $G=\bigcup_{i\in\mathbb{N}} D_i$.
    \item For each $i<j$, $D_j=\bigcup_{\gamma\in D_{j}\cap\Gamma_i}\gamma D_i$.
\end{itemize}

Let $(\alpha_i)_{i\in\mathbb{N}}\subseteq \Sigma$ be the sequence given by $\alpha_i=j\in\Sigma $ when $i\equiv j \pmod{n}$. 
The element $\eta\in\Sigma^{G}$ is defined as follows:

\medskip

{\bf Step 1:} Let $J(0)=\{1_G\}$, and define $\eta(g)=\alpha_1$ for every $g\in \Gamma_1$.

\medskip

{\bf Step 2:}
Define $ J(1)=D_1\setminus \Gamma_1$.
For every $h\in J(1)$ and $\gamma\in \Gamma_2$, 
 $\eta(\gamma h)=\alpha_2$. 

\medskip

{\bf Step m+1:} Consider
\begin{align*}
    J(m)=D_{m}\setminus \bigcup_{i=0}^{m-1} J(i)\Gamma_{i+1}.
\end{align*}

Define $\eta(\gamma h)=\alpha_{m+1}$ for every $h\in J(m)$ and $\gamma\in \Gamma_{m+1}$.

\medskip

This construction defines a Toeplitz array $ \eta$ such that $(\Gamma_i)_{i\in\mathbb{N}}$ is a period structure for $\eta$. 
Therefore, there exists a factor map $\pi: \overline{O_\sigma(\eta)}\to \overleftarrow{G}$, where $\overleftarrow{G}$ is the $G$-odometer associated to $(\Gamma_i)_{i\in\mathbb{N}}$ and also it is the maximal equicontinuous factor of $\overline{O_\sigma(\eta)}$.
Furthermore, when $G$ is amenable and
\begin{align}\label{Gammacondition}
    [\Gamma_i:\Gamma_{i+1}]>\frac{1}{1-1/2^{(1/2)^{i+1}}},\;\;i\in\mathbb{N},
\end{align} 
$\overline{O_\sigma(\eta)}$ has exactly $n$ ergodic measures with zero entropy.

\medskip

The following lemma and corollary were proved in \cite{CeCoGo23}. 
We provide the proofs of these statements to clarify the discussion in Section \ref{Sec-virtually}. 
\begin{lemma}[{\cite[Lemma 4.11]{CeCoGo23}}]\label{auxiliar}
For every $i\geq 1$ and $\gamma\in \Gamma_i$, there exists $l\geq i$ such that $\gamma J(i)\subseteq \Gamma_{l+1}J(l)$. 
\end{lemma}
\begin{proof}
Since $J(i)=D_i\setminus \Per(\eta, \Gamma_i)$, we have $\gamma J(i)\cap \Per(\eta, \Gamma_i)=\emptyset$. This implies that
$$
\gamma J(i)\subseteq \bigcup_{l\geq i}\Gamma_{l+1}J(l).
$$
Let $l=\min\{k\geq i: \gamma J(i)\cap \Gamma_{k+1}J(k)\neq \emptyset\}$.  Let $u\in J(i)$ be such that $\gamma u= \gamma_{l+1}v_l $, for some $v_l\in J(l)$ and $\gamma_{l+1}\in \Gamma_{l+1}$. Since $v_l\in D_l$, there exist $v\in D_i$ and $\gamma'\in \Gamma_i\cap D_l$ such that $v_l=\gamma'v$. 
The relation $\gamma u= \gamma_{l+1}v_l $ implies $v=u$ and $\gamma=\gamma'\gamma_{l+1}'$, for some $\gamma_{l+1}'\in \Gamma_{l+1}$.  Thus, if $s\in J(i)$, then $\gamma s= \gamma'\gamma_{l+1}'s=\gamma's\gamma_{l+1}''$, for some $\gamma_{l+1}''\in \Gamma_{l+1}$. 
This implies that $\gamma s\in \Gamma_{l+1}D_l\subseteq \Gamma_{l+1}J(l) \cup \bigcup_{k=0}^{l-1}\Gamma_{k+1}J(k)$.
The choice of $l$ implies that $\gamma s\in \Gamma_{l+1}J(l)$,  and then $\gamma J(i)\subseteq \Gamma_{l+1}J(l)$.
\end{proof}

\begin{corollary}[{\cite[Corollary 4.12]{CeCoGo23}}]\label{partition}
For every $i\geq 0$ and $\gamma\in \Gamma_i$, there exists $\alpha\in \Sigma$ such that 
$$
\eta(g)=\alpha \mbox{ for every } g\in \gamma J(i).
$$
\end{corollary}
\begin{proof}
The case $i=0$ is trivial. Suppose that $i\geq 1$ and $\gamma\in \Gamma_i$.

From Lemma \ref{auxiliar}, there exists $l\geq i$ such that $\gamma J(i)\subseteq \Gamma_{l+1}J(l)$. 
By the definition of $\eta$ we get $\eta(g)=\alpha_{l+1}$, for every $g\in \gamma J(i)$.   
\end{proof}

 For $x\in \overline{O_\sigma(\eta)}$, let $\pi(x)=(t_i(x)\Gamma_i)_{i\in\mathbb{N}}\in \overleftarrow{G}$ be the image of $x$ in $\overleftarrow{G}$, where $t_i(x)\in D_i$ is a representative of the coset $(\pi(x))_i$ for each $i\in \mathbb{N}$, and $\Aper(x)$ the subset in $G$ given by   \begin{align*}
     \Aper(x)=G\setminus \bigcup_{i\in\mathbb{N}}\Per(x,t_i(x)^{-1}\Gamma_it_i(x)).
     \end{align*}
     \begin{remark}\label{rem-Aper}
The subset $\Aper(x)$ depends only on $\pi(x)$. That is, for $x,y\in\pi^{-1}(\{\pi(x)\})$ it holds $\Aper(x)=\Aper(y)$.
Moreover, by the definition of $\pi:\overline{O_\sigma(\eta)}\to\overleftarrow{G}$ and Lemma \ref{Per-lemma}, we obtain $x(g)=y(g)$, for each $g\in G\setminus \Aper(x)$.
     \end{remark}

\begin{lemma}\label{aper-const}
    
Let $x\in \overline{O_\sigma(\eta)}$ and $\pi(x)=(t_i(x)\Gamma_i)_{i\in\mathbb{N}}$, where $t_i(x)\in D_i$ for each $i\in\mathbb{N}$. 
It holds that  the map $x|_{t_i(x)^{-1}\gamma J(i)}$ is constant for each  $i\in\mathbb{N}$ and  $\gamma\in \Gamma_i$.
That is, there exists $\alpha\in\Sigma$ such that $x(d)=\alpha$ for every  $d\in t_i(x)^{-1}\gamma J(i)$.
In particular,  $x(d)=\alpha$ for every $d\in (t_i(x)^{-1}\gamma D_i)\cap \Aper(x)$.
\end{lemma}
\begin{proof}
    
By definition of $\pi$, $x\in \sigma^{(t_i(x))^{-1}}C_i$, for each $i\in\mathbb{N}$.
Since $\overline{O_\sigma(\eta)}$ is minimal, there exists a sequence $(g_k)_{k\in\mathbb{N}}\subseteq G$ such that $\sigma^{g_k}\eta$ converges to $x$. 
Let $i\in\mathbb{N}$. 
Since ${\sigma^{(t_i(x))^{-1}}}C_i$ is a clopen set, we deduce that there exists $k_i’\in \mathbb{N}$ such that $\sigma^{g_j}\eta\in \sigma^{(t_i(x))^{-1}}C_i$ for every $j\geq k_i’$. 
Consequently, for each $j\geq k_i’$, there exists $\gamma_i^j\in\Gamma_i$ satisfying $g_j=(t_i(x))^{-1}\gamma_i^j$.
 Since $\sigma^{(t_i(x))^{-1}\gamma_i^j}\eta\to x$ when $j\to \infty$, for each $\gamma\in \Gamma_i$ we can pick $k_i\in\mathbb{N}$ (which depends on $\gamma$) with $k_i\geq k_i’$ so that $  (t_i(x))^{-1}\gamma J(i)\subseteq D_{k_i}$.
 Moreover, there exists $t_{k_i}\in\mathbb{N}$ so that $\sigma^{g_j}\eta\in \sigma^{(t_i(x))^{-1}}C_i$ and $\sigma^{g_j}\eta(d)=x(d)$ for every $d\in D_{k_i}$, $j\geq t_{k_i}$.
For $d=(t_i(x))^{-1}\gamma d_i\in (t_i(x))^{-1}\gamma J(i)$, we have $\sigma^{g_j}\eta(d)=\eta(g_j^{-1}d)=\eta((\gamma_i^j)^{-1}t_i(x)d)=\eta((\gamma_i^j)^{-1}\gamma d_i)$. 
Since $\eta|_{(\gamma_i^j)^{-1}\gamma J(i)}$ is constant by the previous corollary, we deduce that $x|_{(t_i(x))^{-1}\gamma J(i)}$ is constant.

Furthermore, for every $\gamma\in \Gamma_i$
\begin{align*}
    (t_i(x))^{-1}\gamma D_i\cap \Aper(x)
    &\subseteq 
    t_i(x)^{-1}\gamma D_i\cap \left(G\setminus\Per(x,t_i(x)^{-1}\Gamma_it_i(x))\right)\\
    &=  t_i(x)^{-1}\gamma D_i\cap (G\setminus \Per(\sigma^{(t_i(x))^{-1}}\eta,t_i(x)^{-1}\Gamma_it_i(x)))\\
    &=  t_i(x)^{-1}\gamma D_i\cap (G\setminus \Per(\sigma^{t_i(x)^{-1}\gamma}\eta,t_i(x)^{-1}\Gamma_it_i(x)))\\
    &=(t_i(x)^{-1}\gamma)(D_i\setminus \Per(\eta,\Gamma_i))\\
    &=t_i(x)^{-1}\gamma J(i).
    \end{align*}

This concludes the proof.
\end{proof}

For infinite countable groups that are direct sum of finite groups we obtain the following result.

 \begin{proposition}\label{prop: sum-finite}
If  $G\cong \oplus_{i=1}^\infty G_i$, where $G_i$ is a finite non-trivial group, $i\in\mathbb{N}$, then  $h^*_{top}(\overline{O_\sigma(\eta)},\sigma, G)=\log(n)$.  
\end{proposition}
\begin{proof}
For each $m\in\mathbb{N}$,    consider $\Gamma_m:=\oplus_{i=m}^\infty G_i$. 
These subgroups are normal and we can choose the fundamental domains to be the groups $D_m=\oplus_{i=1}^{m-1}G_i$.
Moreover, note that $d^{-1}\gamma=\gamma d^{-1}$ and this in turn implies that  $d^{-1}\gamma D_m=\gamma D_m$, for every $d\in D_m$ and $\gamma\in \Gamma_m$, $m\in\mathbb{N}$. 
The Toeplitz subshift $\overline{O_\sigma(\eta)}$ previously   constructed  has exactly $n$ ergodic measures as the group $G$ is amenable, and it is free since $\bigcap_{i=1}^\infty \Gamma_i=\{1\}$. 
Let $x\in \overline{O_\sigma(\eta)}$ be a non-Toeplitz element and $\pi(x)=(t_i(x)\Gamma_i)_{i\in\mathbb{N}}$.
We claim that there exists $\alpha\in \Sigma$ such that for every element $g\in \Aper(x)$ it holds $x(g)=\alpha$.
Indeed, assume that $g,h\in \Aper(x)$ are such that $x(g)\neq x(h)$. 
There exists $m_0\in\mathbb{N}$ in such a way that $g,h\in D_{m_0}$. 
Applying Lemma \ref{aper-const} and the fact that $t_{m_0}^{-1}(x)D_{m_0}=D_{m_0}$, we conclude that $x(g)=x(h)$, which is a contradiction.
Then, $|\pi^{-1}(\pi(x))|=n$ and we deduce that $h^*_{top}(X,\varphi,G)=\log(n)$ by using Propositions \ref{It-to-Reg}, \ref{medidas-IT} and \ref{htop*}.
\end{proof}

\subsection{Abelian finitely generated groups}\label{Abelianfinite}
Let $G$ be a torsion-free finitely generated abelian group  of rank $r>1$, i.e., $G=\mathbb{Z}^r$. 
Let $S=\{\pm \tilde{e_j}:1\leq j\leq r\}$ be the canonical symmetric set of generators of $G$.
That is, $\tilde{e_j}$ is the element
$\tilde{e_j}=(0,\dots,0,1,0,\dots,0)$, where the element $1$ is in the $j$-th position.
 
For $i\in\mathbb{N}$, let $\Gamma_i$ be the subgroup given by
\begin{align}\label{Gammas}
    \Gamma_i=\langle  p^i_j\tilde{e_j}:1\leq j\leq r\rangle_G,
\end{align} 
where $(p^i_j)_{i\in\mathbb{N}}$ is a strictly increasing sequence of natural numbers such that $p_j^i\to \infty$ when $i\to \infty$ and $p_j^i|p_{j}^{i+1}$ for every $1\leq j\leq r$.
Up to taking subsequences, if necessary, we can assume that $p_j^i>2i+1$ and that $(\Gamma_i)_{i\in\mathbb{N}}$ satisfies (\ref{Gammacondition}).
Now, for each $i\in\mathbb{N}$ consider $D_i'\subseteq G$ as some set of the form
\begin{align}\label{form D}
    D_i'=\{(x_1,\dots, x_r)\in \mathbb{Z}^r: -q_{1,j}^i\leq x_j<q_{2,j}^i, 1\leq j\leq r\},  
\end{align}
for some $q_{t,j}^i> i$, $t\in\{1,2\}$ such that $q_{1,j}^i+q_{2,j}^i=p_j^i$, $1\leq j\leq r$ and $D_i'\subseteq D_{i+1}'$.

Observe that $(D_i')_{i\in\mathbb{N}}$ is a sequence of finite subsets of $G$ such that $D_i'$ is a fundamental domain of $G/\Gamma_i$ and  $G=\bigcup_{i\in\mathbb{N}}D_i'$.
The proofs of Lemma 3 and Lemma 4 in \cite{CoPe14} guarantee the existence of a sequence of natural numbers $(n_i)_{i\in\mathbb{N}}\subseteq \mathbb{N}$ and a sequence of finite subsets of $G$, denoted by $(D_i)_{i\in\mathbb{N}}$, such that  
\begin{itemize}
\item $D_1=D_1'$.
    \item $D_i=\bigcup_{\gamma\in D_{n_i}'\cap \Gamma_{n_{i-1}}} \gamma+ D_{i-1}=\bigcup_{\gamma\in D_i\cap \Gamma_{n_{i-1}}} \gamma +D_{i-1}$, for $i\geq 2$.

\item $D_i'\subseteq D_i$. 
In particular, $G=\bigcup_{i\in\mathbb{N}}D_i$.

 \item $D_i$ is a fundamental domain of $G/\Gamma_{n_i}$.

\item $\overline{0}=(0,\dots,0)\in D_i\subseteq D_{i+1}$ for each $i\in\mathbb{N}$.

\item $(D_i)_{i\in\mathbb{N}}$ is a F\o lner sequence for $G$.

\end{itemize}
Notice that $D_i$ has the form given in (\ref{form D}), $i\in\mathbb{N}$.
Moreover, if necessary, we can take a subsequence of $(\Gamma_i)_{i\in\mathbb{N}}$ such that $n_i=i$ for each $i\in\mathbb{N}$.

For each $m\in \mathbb{N}$ and $d\in\mathbb{Z}^r$, denote by $B(d,m)$ the hypercube  in $\mathbb{Z}^r$ which is centered at $d=(d_1,\dots,d_r)$ and whose sides measure $2m+1$, i.e.,
\begin{align}\label{rectangle}
    B(d,m)=\{(z_1,\dots, z_r)\in\mathbb{Z}^r: -m\leq z_i-d_i\leq m, 1\leq i\leq r\}.
\end{align}
Define $b(m):=|B(\overline{0},m)|$.
Then, for every $m\in\mathbb{N}$,
\begin{align*}
    \dfrac{b(m+1)}{b(m)}=\left(1+\dfrac{2}{2m+1}\right)^r=1+\dfrac{1}{2m+1}y',
\end{align*}
for some $y'$ depending on $m$ which is bounded.
Note that $\frac{1}{2m+1}y'$ goes to $0$ as $m$ goes to infinity.
Therefore, for every $\varepsilon>0$ and $s\in\mathbb{N}$ there exists $m_0\in\mathbb{N}$ such that for each $m\geq m_0$,
\begin{align}\label{Boxes}
    \dfrac{b(m+s)}{b(m)}<1+{\varepsilon}.
\end{align}

\begin{lemma}\label{1/2r}
    Let $n,s\in\mathbb{N}$,  and $d\in D_n$.
    If $\gamma\in \Gamma_{n+s}$ is such that $d\in\gamma+ D_{n+s}$, then 
    \begin{align*}
        \frac{1}{2^r}b(s)\leq |B(d,s)\cap \gamma +D_{n+s}|.
    \end{align*}
\end{lemma}
\begin{proof}
    Let $\gamma\in \Gamma_{n+s}$ be such that $d\in \gamma+ D_{n+s}$.
Since $D_{n+s}$ has the form described in \eqref{form D}, it holds that at least one of the sets $d+O$ is contained in $ \gamma +D_{n+s}$, where 
    \begin{align*}
        O=\{(x_1,\dots, x_r)\in\mathbb{Z}^r: 0\leq \epsilon_jx_j\leq s, 1\leq j\leq r\},
    \end{align*}
    for some $\epsilon_j\in\{-1,1\}$, $1\leq j\leq r$. 
    The set $O$  contains a $\frac{1}{2^r}$ part of the hypercube $[-s,s]^r$ in $\mathbb{Z}^r$. 
    Thus, $|O\cap B(\overline{0},s)|\geq \frac{|B(\overline{0},s)|}{2^r}$, as we desired.
\end{proof}
 
For $n\geq 2$, let  $\Sigma=\{1,2,\dots, n\}$,  $\eta\in \Sigma^G$ be the Toeplitz array presented in Section \ref{Sec-3} associated to the sequence $(\Gamma_i)_{i\in\mathbb{N}}$ given by \eqref{Gammas} and $(D_i)_{i\in\mathbb{N}}$ the sequence of fundamental domains for $G/\Gamma_i$ previously presented.  
The subshift $X=\overline{O_\sigma(\eta)}$ denotes the Toeplitz subshift associated to $\eta$.

\medskip

Recall that for $x\in X$ we denote $\pi(x)=(t_i(x)+\Gamma_i)_{i\in\mathbb{N}}$ as the image under the map $\pi:X\to \overleftarrow{G}$ from $X$ to its maximal equicontinuous factor $\overleftarrow{G}$, where $t_i(x)\in D_i$ for each $i\in\mathbb{N}$.

Let $i\in\mathbb{N}$.
For each $\zeta\in\Gamma_i$,  denote by $T_\zeta(x)$ the union of sets of the form $-t_{i+j}(x)+\gamma_{i+j}^{\zeta}+D_{i+j}$, where for each $j\geq 1$, $\gamma_{i+j}^{\zeta}$ is the only element in $\Gamma_{i+j}$ such that $-t_{i+j}(x)+\gamma_{i+j}^\zeta+D_{i+j}\supseteq -t_{i+(j-1)}(x)+\gamma_{i+(j-1)}^\zeta+ D_{i+(j-1)}$ and $\gamma_{i}^\zeta=\zeta$, i.e., 
\begin{align*}
    T_\zeta(x)=\bigcup_{j\in\mathbb{N}}-t_{i+j}(x)+\gamma_{i+j}^\zeta+ D_{i+j}.
\end{align*}

\begin{remark}
If $\pi(x)=\pi(y)$, then $-t_{i}(x)+\zeta +D_i=-t_i(y)+\zeta+ D_i$ for every $i\in\mathbb{N}$ and $\zeta\in \Gamma_i$. 
Hence, $T_\zeta(x)=T_\zeta(y)$.
\end{remark}

\begin{lemma}\label{decomG}
   For each $x\in X$, there exist $\beta\leq 2^r$ and $\zeta_{n_i}\in\Gamma_{n_i}$ for $1\leq i\leq \beta$, such that 
    \begin{align*}
        G=\bigsqcup_{i=1}^{\beta} T_{\zeta_{n_i}}(x).
    \end{align*}
\end{lemma}
\begin{proof}
Since  $G=-t_n(x)+\Gamma_n +D_n$ for each $n\in\mathbb{N}$, there exist finite elements $\gamma_1,\gamma_2,\dots,\gamma_s\in\Gamma_n$ such that $B(\overline{0},n)\subseteq \bigcup_{i=1}^{s}(-t_n(x)+\gamma_i +D_n)$ and $(-t_n(x)+\gamma_i +D_n)\cap B(\overline{0},n)\neq\emptyset$.
We also remark that they are the only elements that satisfy the previous condition.
Therefore, we conclude  
\begin{align*}
    G=\bigsqcup_{i=1}^{\beta}T_{\zeta_{n_i}}(x),
\end{align*}
for some $\beta\in\mathbb{N}\cup\{\infty\}$ and $\zeta_{n_i}\in\Gamma_{n_i}$, $1\leq i\leq \beta$.

Now, we prove that $\beta \leq 2^r$.
By contradiction, assume that $\beta>2^r$.
Thus, for some $n_0\in\mathbb{N}$ there exist $\beta'\in\mathbb{N}$, with $\beta \geq \beta'>2^r$ and $\zeta_1,\zeta_2,\dots,\zeta_{\beta'}\in\Gamma_{n_0}$ such that the sets $T_{\zeta_i}(x), i\in\{1,2,\dots, \beta'\}$, are  pairwise disjoint, $B(\overline{0},n_0)\subseteq \bigsqcup_{i=1}^{\beta'}-t_{n_0}(x)+\zeta_i+ D_{n_0}$ and $T_{\zeta_i}(x)\cap B(\overline{0},n_0)\neq \emptyset$ for each $1\leq i\leq \beta'$.
Therefore,  we have that $-t_{n_0}(x)+\zeta_i+d_i\in B(\overline{0},n_0)$ for some $d_i\in D_{n_0}$, $1\leq i\leq \beta'$.
Consequently, $B(-t_{n_0}(x)+\zeta_i +d_i,s)\subseteq B(\overline{0},n_0+s)$ for every $s\in\mathbb{N}$.
Hence, we have 
\begin{align*}
\bigsqcup_{i=1}^{\beta'}    B(-t_{n_0}(x)+\zeta_i +d_i,s)\cap (-t_{n_0+s}(x)+\gamma_{n_0+s}^{\zeta_i}+ D_{n_0+s})\subseteq B(\overline{0},n_0+s),
\end{align*}
for all $s\in\mathbb{N}$, which implies  

\begin{align} \label{alpha-prime}
\sum_{i=1}^{\beta'}   |B(-t_{n_0}(x)+\zeta_i+ d_i,s)\cap (-t_{n_0+s}(x)+\gamma_{n_0+s}^{\zeta_i} +D_{n_0+s})|\leq b(n_0+s).
\end{align}

Using  (\ref{Boxes}) with $\varepsilon<\frac{\beta'}{2^r}-1$ and $s=n_0$, we obtain that there exists $s\geq 1$ such that 
\begin{align*}
    \dfrac{b(n_0+s)}{b(s)}<1+\varepsilon.
\end{align*}
Combining this with (\ref{alpha-prime}) and Lemma \ref{1/2r}  we obtain
\begin{align*}
    \dfrac{\beta'}{2^r}b(s)<(1+\varepsilon)b(s),
\end{align*}
which is a contradiction with the choice of $\varepsilon$. 
Thus, $\beta\leq 2^r$ and we conclude the lemma.
\end{proof}

\begin{lemma}\label{Const-T}
Let $x\in X$ be a non-Toeplitz element and $i\in\mathbb{N}$. 
For every $\zeta\in\Gamma_i$ there exists $\alpha\in\Sigma$ such that $x(d)=\alpha$ for every  $d\in T_\zeta(x)\cap\Aper(x)$.
\end{lemma}
\begin{proof}
   Let $d,d'\in T_\zeta(x)\cap\Aper(x)$. 
    By the construction of $T_\zeta(x)$ there exists $j_0\geq 0$ such that $d,d'\in -t_{i+j_0}(x)+\gamma_{i+j_0}^\zeta +D_{i+j_0}$.
    Using Lemma \ref{aper-const}, we obtain that $x(d)=x(d')$.
\end{proof}

\begin{proposition}\label{Regio-prox-bound}
It holds that $|\pi^{-1}(\{\pi(x)\})|\leq |\Sigma|^{2^r}$ for every $x\in X$.
\end{proposition}
\begin{proof}Let $x\in X$ and $\overline{g}=\pi(x)$. 
If $x$ is a Toeplitz element, we already know that $|\pi^{-1}(\overline{g})|=|\{x\}|=1$.
Assume $x\in X$ is a non-Toeplitz element. 
Lemma \ref{decomG} guarantees a decomposition of $G$ using at most $2^r$ sets of the form $T_\zeta$ for some $\zeta\in\Gamma_i$, $i\in\mathbb{N}$.
Remark \ref{rem-Aper} implies that for $y,z\in\pi^{-1}(\{\overline{g}\})$, it holds that $y(g)=z(g)$ for each  $g\in G\setminus\Aper(x)$.
On the other hand, Lemma \ref{Const-T} implies that for each $y\in\pi^{-1}(\{\overline{g}\})$, $y|_{T_\zeta(x)\cap \Aper(x)}$ is constant.
Thus,  $\pi^{-1}(\{\overline{g}\})$ has at most $|\Sigma|^{2^r}$ elements.
\end{proof}

The previous proposition and Proposition \ref{It-to-Reg} imply the following corollary.
 \begin{corollary}\label{IT-empty}
    If $s>|\Sigma|^{2^r}$, then
    $\IN_s(X)\setminus \triangle^{(s)}(X)=\emptyset$.
 \end{corollary}

\begin{corollary}\label{CoroZr}
    It holds that  $$\log(m)\leq h_{top}^*(X,\sigma,\mathbb{Z}^r)\leq\log(m^{2^r}),$$ $\IT_{m}(X)\setminus \triangle^{(m)}(X)\neq \emptyset$  and $\IT_{m^{2^r}+1}(X)\setminus \triangle^{(m^{2^r}+1)}(X)=\emptyset$.
\end{corollary}
\begin{proof}
    As a consequence of (\ref{Gammacondition}), $X$ has exactly $m$ ergodic measures.
    Therefore, $\IT_m(X)\setminus\triangle^{(m)}(X)\neq \emptyset$ by Proposition \ref{medidas-IT}.
   We conclude the proof by using Corollary \ref{IT-empty} and Proposition \ref{htop*}.
\end{proof}

The statement below is a version of Theorem \ref{theo-main} for $G=\mathbb{Z}^r$, and it is a direct consequence of Corollary \ref{CoroZr}.

\begin{corollary}\label{maintheo:Zr}
    Let $r\geq 2$.
    For every $m\geq 2$, there exists a free minimal subshift  $X\subseteq\{1,\dots,m\}^{\mathbb{Z}^r}$ with zero entropy such that  $\log(m)\leq h_{top}^*(X,\sigma,\mathbb{Z}^r)\leq\log(m^{2^r})$, $\IT_{m}(X)\setminus \triangle^{(m)}(X)\neq \emptyset$  and $\IT_{m^{2^r}+1}(X)\setminus \triangle^{(m^{2^r}+1)}(X)=\emptyset$.
\end{corollary}
 
\subsection{Virtually $\mathbb{Z}^r$ groups}\label{Sec-virtually}

Let $G$ be a group with a finite index normal subgroup $G'$ that is isomorphic to $\mathbb{Z}^r$, for some $r\in\mathbb{N}$. 
Consider $(\Gamma_i)_{i\in\mathbb{N}}$
the sequence of finite index subgroups of $\mathbb{Z}^r$ subgroups given in \eqref{Gammas} with the extra property given in \eqref{Gammacondition}.
Let $(D_i)_{i\in\mathbb{N}}$ and $(D_i')_{i\in\mathbb{N}}$ be the sequences defined in Section \ref{Abelianfinite}.
We slightly abuse notation by assuming that the sequences $(\Gamma_n), (D_i)_{i\in\mathbb{N}}$ and $(D_i')_{i\in\mathbb{N }}$ are sequences of  $G'$ instead of $\mathbb{Z}^r$.
Since $G'$ is a subgroup of finite index in $G$, there exists a finite subset $R\subseteq G$ such that $R$ is a set of representatives for $G/G'$ that contains $1_G$.

 The sequence 
 $(D_iR)_{i\in\mathbb{N}}$, where $D_iR=\{dr\in G: d\in D_i, r\in R\} $, is the adequate sequence for the construction in this case. 
 This sequence has the following properties: For each $i\geq 1$,
\begin{itemize}
    \item $1_G\in D_iR$ and $D_iR$ is a set of representatives for $\Gamma_i\backslash G$, the set of right cosets of $\Gamma_i$ in $G$.
    \item $D_iR\subseteq D_{i+1}R$.
    \item $\bigcup_{i\in\mathbb{N}}D_iR=G$.
    \item $D_{i+1}R=\bigcup_{\gamma\in\Gamma_i\cap D_{i+1}}\gamma D_iR$.
\end{itemize}
Moreover, it satisfies the  left F\o lner condition.

\begin{lemma}
  It holds that $(D_iR)_{i\in\mathbb{N}}$ is a left F\o lner sequence for $G$, i.e., for every $g\in G$ 
 \begin{align*}
     \lim_{i\to\infty}\dfrac{|D_iRg\setminus D_iR|}{|D_iR|}=0.
 \end{align*}
\end{lemma}
\begin{proof}
    Notice that for each $g\in G$, there exist $d\in G'$ and $s\in R$ such that $g=ds$.
    Let $r\in R$.
    As $G'$ is normal in $G$, there exists $d_r\in G'$ such that $rd=d_rr$. 
    Therefore, $rg=d_rrs=d_rt_rs_r$, where  $rs=t_rs_r$ for some $t_r\in G'$ and $s_r\in R$.
    Thus, for any $i\in\mathbb{N}$
    \begin{align*}
        \dfrac{|D_iRg\setminus D_iR|}{|D_iR|}=&\dfrac{\left|\left(\bigcup_{r\in R}D_id_rt_rs_r\right)\setminus D_iR\right|}{|D_iR|}\\
        \leq &\sum_{r\in R}\dfrac{\left|d_rt_rD_is_r\setminus D_iR\right|}{|D_iR|}\\
        \leq &\sum_{r\in R}\dfrac{\left|d_rt_rD_iR\setminus D_iR\right|}{|D_iR|}\\
        = &\sum_{r\in R}\dfrac{\left|d_rt_rD_i\setminus D_i\right|}{|D_i|}.
    \end{align*}
    The proof concludes using the fact that $(D_i)_{i\in\mathbb{N}}$ is a F\o lner sequence for $G'$.
\end{proof}
Given $m\in\mathbb{N}$, denote $\Sigma=\{1,\dots,m,\beta\}$,  where $\beta$ is a symbol that is not in $\{1,\dots, m\}$. 
 We construct $\eta\in\Sigma^G$ in such a way that its associated subshift, $\overline{O_\sigma(\eta)}$, satisfies the conditions in Theorem \ref{theo-main}.
This construction is analogous to that presented in \cite{CeCoGo23}, except that in this case $\Gamma_i$ is not necessarily normal in $G$ and the tiling condition for the sequence $(D_iR)_{i\in\mathbb{N}}$ differs from that assumed in \cite{CeCoGo23}.
Let $(\alpha_i)_{i\in\mathbb{N}}$ be the sequence defined by $\alpha_i=j\in\{1,\dots,m\}$ if $i\equiv j \pmod{m}$. 
We define $\eta\in\Sigma^G$ as follows:

\medskip

\textbf{Step 1:} Consider $J(0)=\{1_G\}$.
Define $\eta(g)=\alpha_1$ for every $g\in \Gamma_1J(0)$, and $\eta(g)=\beta$ for each $g\in \Gamma_1(R\setminus \{1_G\})$.

\medskip

\textbf{Step 2:} Consider $J(1)=D_1\setminus J(0)\Gamma_1$.
Define $\eta(g)=\alpha_2$ for every $g\in \Gamma_2J(1)R$.

\medskip

\textbf{Step $\bm{n+1}$:} Let 
\begin{align*}
    J(n)=D_n\setminus \bigcup_{i=1}^{n-1} \Gamma_{i+1}J(i).
\end{align*}
Define $\eta(g)=\alpha_{n+1}$ for each $g\in \Gamma_{n+1}J(n)R$.

\medskip

This construction yields a Toeplitz array in $\Sigma^G$, which we denote by $\overline{O(\eta)}$ from now on.

\subsubsection{The subshift $\overline{O(\eta)}$ has $m$ ergodic measures  and zero entropy} This subsection is similar to the last section in \cite{CeCoGo23}, we write the proofs for our case since they differ from \cite{CeCoGo23} in the fact that $\Gamma_i$ is not necessarily normal in $G$ and the tiling condition for the sequence $(D_iR)_{i\in\mathbb{N}}$ is not the same.
We need some technical lemmas that describe the behavior of the sets $J(n)$ and the period set of $\eta\in \Sigma^G$ as constructed before.
The following lemma is proven similarly to \cite[Lemma 4.2]{CeCoGo23}
\begin{lemma}\label{Lemma:J's}
For every $n\in\mathbb{N}$, it holds that 
    $$J(n)=\bigcup_{\gamma\in (D_{n}\cap \Gamma_{n-1})\setminus\{1_G\}}\gamma J(n-1).$$
\end{lemma}
\begin{lemma}\label{Lemma:Per}
    For every $n\geq 1$ it is true that
    $$\Per({\eta},\Gamma_n)=\bigcup_{i=1}^{n-1}\Gamma_{i+1}J(i)R.$$
 Furthermore,  $\Per({\eta},\Gamma_1,\beta)=\Gamma_1(R\setminus\{1_G\})$ and $\Per({\eta},\Gamma_1,\alpha_1)=\Gamma_1$.
\end{lemma}
\begin{proof}
From the construction, it follows that $\bigcup_{i=1}^{n-1}\Gamma_{i+1}J(i)R\subseteq \Per(\eta, \Gamma_n)$.
Now, if $g \in \Per(\eta, \Gamma_n)\setminus \bigcup_{i=1}^{n-1}\Gamma_{i+1}J(i)R$ we can consider $\gamma\in \Gamma_n$, $r\in R$ and $d\in D_n$ so that $g=\gamma dr$.
Moreover, we have $dr\in (D_nR\cap\Per(\eta, \Gamma_n))\setminus \bigcup_{i=1}^{n-1}\Gamma_{i+1}J(i)R$ and it in turn implies $d\in J(n)$.
This is a contradiction with the fact that for every $\gamma\in (\Gamma_n\cap D_{n+1})\setminus\{1_G\}$ we have that $\eta(\gamma dr)=\alpha_{n+2}$ and $\eta(dr)=\alpha_{n+1}$.
The second part of the lemma follows analogously.
\end{proof}
\begin{lemma}\label{Lemma:R}
For every $n\geq 2$ and $\alpha\in\{1,\dots, m\}$, it holds
    $$\Per(\eta,\Gamma_n,\alpha)\cap D_nR=(\Per(\eta,\Gamma_n,\alpha)\cap D_n)R.$$
\end{lemma}
\begin{proof}
Let $n\geq 2$ and $\alpha\in \{1,\dots,m\}$. 
Consider $dr\in \Per(\eta,\Gamma_n,\alpha)\cap D_nR$, where $d\in D_n$ and $r\in R$.
Lemma \ref{Lemma:Per} implies that $d\in J(i)$ for some $i\leq n-1$.
Furthermore, we conclude that $\eta(\gamma dr)=\eta(\gamma d)=\alpha$ for each $\gamma\in\Gamma_{i+1}$, in particular, for each $\gamma\in \Gamma_n$.
Therefore, $d\in \Per(\eta,\Gamma_n,\alpha)\cap D_n$ and consequently, $\Per(\eta,\Gamma_n,\alpha)\cap D_n R\subseteq (\Per(\eta,\Gamma_n,\alpha)\cap D_n) R $.
For the converse, if $dr\in (\Per(\eta,\Gamma_n,\alpha)\cap D_n)R$, then $d\in J(i)$ for some $i\leq n-1$ and the definition of $\eta$ implies that $\eta(\gamma dr)=\eta(\gamma d)$ for every $\gamma \in\Gamma_n$.
Therefore,  $dr\in \Per(\eta, \Gamma_n,\alpha)\cap D_n R$.
\end{proof}

\begin{proposition}
    The sequence $(\Gamma_i)_{i\in\mathbb{N}}$ is a period structure for ${\eta}$. 
\end{proposition}
\begin{proof}
    We use induction on $i$:
    For $i=1$, let us assume that for some $g\in G$  
    \begin{align}\label{essential}
        \Per({\eta},\Gamma_1,\alpha)\subseteq \Per(\sigma^g{\eta},\Gamma_1,\alpha), \mbox{ for each } \alpha\in\Sigma.  
    \end{align}
    There exist $d\in G'$ and $r\in R$ such that $g=dr$. 
    Assume that $r\neq 1_G$. 
    Using that
 $\Per({\eta},\Gamma_1, \beta)=\Gamma_1(R\setminus\{1_G\})$, we obtain  ${\eta}(g^{-1}\gamma r')=\beta$ for every  $r'\in R\setminus\{1_G\}$ and $\gamma\in \Gamma_1$.
 In particular, ${\eta}(r^{-1}d^{-1}r)=\beta$. 
 By definition of $\eta$,  $r^{-1}dr\in \Gamma_1(R\setminus\{1_G\})$, which is a contradiction since $r^{-1}dr\in G'$.
 Thus, $r=1_G$.
 From \eqref{essential} and the fact that $G'$ is commutative, we  deduce $g^{-1}\in\Per({\eta},\Gamma_1,\alpha_1)$ and consequently, $g^{-1}\in\Gamma_1$, as we desired.

\medskip

 Assume the result is true for $i-1\geq 1$.
 Let $g\in G$ be such that $\Per({\eta},\Gamma_i, \alpha)\subseteq \Per(\sigma^g{\eta},\Gamma_i,\alpha)$ for each $\alpha\in\{1,\dots,m,\beta\}$.
 For each $\gamma_{i-1}\in\Gamma_{i-1}$, there exist $\gamma\in D_i\cap \Gamma_{i-1}$ and $\gamma_i\in \Gamma_i$ so that $\gamma_{i-1}=\gamma \gamma_i$. 
 If $w\in\Per({\eta},\Gamma_{i-1},\alpha)$, then $\gamma w\in\Per({\eta},\Gamma_{i-1},\alpha)\subseteq\Per(\eta,\Gamma_i,\alpha)\subseteq  \Per(\sigma^g{\eta},\Gamma_i,\alpha)$.
 Therefore, $\alpha=\sigma^g{\eta}(\gamma w)=\sigma^g{\eta}(\gamma \gamma_i w)=\sigma^g{\eta}(\gamma_{i-1} w)$, it means that $w\in\Per(\sigma^g{\eta},\Gamma_{i-1},\alpha)$.
 Using the hypothesis of induction we conclude that $g\in \Gamma_{i-1}$.

If $g\in\Gamma_{i-1}\setminus \Gamma_i$, there exists $\gamma\in (\Gamma_{i-1}\cap D_i)\setminus\{1_G\}$ and $\gamma_i\in\Gamma_i$ such that $g^{-1}=\gamma \gamma_i$.
If  $h\in J(i-1)\subseteq \Per(\eta,\Gamma_i,\alpha_i)\subseteq \Per(\sigma^g\eta,\Gamma_i,\alpha_i)$, then 
 $\sigma^g{\eta}(\gamma' h)={\eta}(g^{-1}\gamma' h)=\alpha_i$ for every $\gamma'\in\Gamma_i$.
In particular, we have ${\eta}(\gamma h)=\alpha_i$. 
But this gives us a contradiction with the fact that $\gamma h\in J(i)\subseteq \Per({\eta},\Gamma_{i+1},\alpha_{i+1})$ and $\alpha_i\neq \alpha_{i+1}$.
This concludes the proof.
\end{proof}  

Consider the (right) $G$-odometer given by 
\begin{align*}
    \overleftarrow{G}=\{(\Gamma_n g_n)\in\prod_{n\in\mathbb{N}} \Gamma_n\backslash G : \Gamma_{n}g_{n+1}=\Gamma_ng_n, \mbox{ for each }n\in\mathbb{N}\},
\end{align*}
where $\Gamma_n\backslash G$ is the set of right cosets of $\Gamma_n$ in $G$, $n\in\mathbb{N}$. 
The left action $\varphi$ of $G$ on  $\overleftarrow{G}$ is given by $\varphi^g((\Gamma_n g_n)_{n\in\mathbb{N}})=(\Gamma_n g_n g^{-1})_{n\in\mathbb{N}} $ for each $g\in G$ and $(\Gamma_n g_n)_{n\in\mathbb{N}}\in\overleftarrow{G}$.
Recall that the dynamical system $(\overleftarrow{G}, \varphi, G)$ is uniquely ergodic. 
We denote this invariant measure by $\mu$.
The previous proposition and \cite[Proposition 7]{CoPe08} guarantee that $(\overleftarrow{G},\varphi, G)$ is the maximal equicontinuous factor of $\overline{O_\sigma(\eta)}$ with factor map given by $\pi:\overline{O_\sigma(\eta)}\to \overleftarrow{G}$, $\pi(x)=(\Gamma_n g_n)_{n\in\mathbb{N}}$ if and only if $x\in \sigma^{g_n^{-1}}C_n$ for every $n\in\mathbb{N}$, where  
    $$
C_n=\{ x\in \overline{O_\sigma(\eta)}: \Per(x,\Gamma_n, \alpha)=\Per(\eta, \Gamma_n, \alpha), \mbox{ for every } \alpha\in \Sigma\},
$$
and $\{\sigma^{v^{-1}}C_n: v\in D_nR\}$ is a clopen partition of $\overline{O_\sigma(\eta)}$ (see \cite{CoPe08}).

Now, we prove that $\overline{O_\sigma(\eta)}$ has exactly $m$ ergodic measures. 
For that, we use the procedure used in \cite{CeCoGo23}.

For $n\geq 1$, define $\eta_n\in \Sigma^G$ as
$$
\eta_n(\gamma D_nR)=\eta(D_nR), \mbox{ for every } \gamma\in \Gamma_n,
$$
that is, $\eta_n(\gamma g)=\eta(g)$ for every  $\gamma \in \Gamma_n$ and $g\in D_nR$.
Thus we have $\sigma^{\gamma}(\eta_n)=\eta_n$, for every $\gamma\in \Gamma_n$, which implies that
$
O_{\sigma}(\eta_n)=\{\sigma^{u^{-1}}(\eta_n): u\in D_nR\}. 
$ 
From this, we can define the following periodic measures on $\Sigma^G$,
$$
\mu_n=\frac{1}{|D_nR|}\sum_{u\in D_nR}\delta_{\sigma^{u^{-1}}(\eta_n)}.  
$$

Let $i\in \{1,\dots, m\}$ and $[i]$ be the subset of all $x\in \Sigma^G$ such that $x(1_G)=i$.

For every $n$ such that $n+1\equiv i \;(\bmod\; m)$ we have
$$
\mu_n([i])=\frac{|J(n)R|+|\Per(\eta,\Gamma_n,i)\cap D_nR|}{|D_nR|}\geq \frac{|J(n)R|}{|D_nR|}=1-\dfrac{|D_n\cap \Per(\eta,\Gamma_n)|}{|D_n|}, 
$$
$$
\mu_n([j])=\frac{|\Per(\eta, \Gamma_n,j)\cap D_nR|}{|D_nR|}\leq\dfrac{|D_n\cap \Per(\eta,\Gamma_n)|}{|D_n|}, \mbox{ for } j\in \{1,\dots, m\}\setminus\{i\},
$$
\begin{align*}
    \mu_n([\beta])=\dfrac{|\Per(\eta,\Gamma_n,\beta)\cap D_nR|}{|D_nR|}=\dfrac{1}{|D_1|}\left(1-\dfrac{1}{|R|}\right)\leq d.
\end{align*}
Notice that $d_n=\frac{|D_nR\cap \Per(\eta, \Gamma_n)|}{|D_nR|}=\frac{|D_n\cap \Per(\eta, \Gamma_n)|}{|D_n|}$ define an increasing sequence converging to some $d\in [0,1]$.
The same applies for the sequence $(d_{n,j})_{n\in\mathbb{N}}$, where  $d_{n,j}=\frac{|D_nR\cap \Per(\eta, \Gamma_n,j)|}{|D_nR|}$.
This implies that every accumulation point $\mu$ of  $(\mu_{i+sm-1})_{ s\in\mathbb{N}}$ satisfies 
\begin{align*}
\mu([i])&=1-d+\lim_{n\to \infty}\frac{|\Per(\eta, \Gamma_n,i)\cap D_n|}{|D_n|}\geq 1-d,\\
\mu([j])&=\lim_{n\to \infty}\frac{|\Per(\eta, \Gamma_n,j)\cap D_nR|}{|D_nR|}=t_j\leq d \mbox{ for every } j\in \{1,\dots, m\}\setminus\{i\}.
\end{align*}

The next proposition implies that $\eta$ is irregular and is proved in the exact manner as \cite[Proposition 4.5]{CeCoGo23} thanks to Lemma \ref{Lemma:R}. 
After this, we can guarantee that $\overline{O_\sigma(\eta)}$ has exactly $m$ ergodic measures.
\begin{proposition}[{\cite[Proposition 4.5]{CeCoGo23}}]\label{regular}
For the Toeplitz array $\eta$ defined above we have   
$$
1-d_{n+1}=\left(1-\frac{1}{|D_{1}|}\right)\prod_{j=1}^n \left(1-\frac{|D_j|}{|D_{j+1}|}\right), \mbox{ for every } n\in \mathbb{N}.
$$
This implies that  $d<1-d$.
\end{proposition}

\begin{remark}\label{different-measures}{\rm The previous proposition implies that $\overline{O_\sigma(\eta)}$ has at least $m$ different invariant measures $\nu_1,\dots,\nu_m$ over $\Sigma^G$, where $\nu_j$ is a limit point of the sequence $(\mu_{j+sm-1})_{s\in\mathbb{N}}$, for each $j\in\{1,\dots, m\}$.
Furthermore, if $\mu$ is an accumulation point of $(\mu_n)_{n\in\mathbb{N}}$ then there exists $i\in \{1,\dots, m\}$ such that $\mu([j])=\nu_i([j])$ for every $j\in\{1,\dots, m,\beta\}$.
Taking subsequences of $(\Gamma_n)_{n\in\mathbb{N}}$, we can assume that $(\mu_{j+sm-1})_{ s\in\mathbb{N}}$ converges to $\nu_j$, for every $1\leq j\leq m$. 
In other words, we can assume that $\nu_1,\dots, \nu_m$ are the unique limit points of $(\mu_n)_{n\in\mathbb{N}}$.

Note that the limit points $\nu_1,\dots, \nu_m$ are supported on $\overline{O_\sigma(\eta)}$.
Indeed,  let  $U\subseteq \Sigma^G$ be a cylinder given by fixing the coordinates of its points in a finite set $F\subseteq G$.
Observe that
$$
\mu_n(U)=\frac{|\{ v\in \partial_F D_nR: \sigma^{v^{-1}}\eta_n\in U\} |}{|D_nR|}+\frac{|\{ v\in D_nR\setminus \partial_F D_nR: \sigma^{v^{-1}}\eta_n\in U\} |}{|D_nR|},
$$
where $\partial_F D_nR=\{v\in D_nR: vF\not\subseteq D_nR\}$.
Since $(D_nR)_{n\in\mathbb{N}}$ is a left F\o lner sequence, the first term of the sum goes to zero as $n$ goes to $\infty$.
This implies that $\nu_i(U)>0$ for some $1\leq  i \leq m$, only if $U$ intersects the orbit of $\eta$.
Thus, the measures $\nu_1,\dots, \nu_m$ are supported on $\overline{O_{\sigma}(\eta)}$.     

}
\end{remark}
At this moment, we have that $\nu_1,\dots,\nu_m$ are different invariant measures on $\overline{O_\sigma(\eta)}$.
In the following we prove that they are the only ergodic measures for the system.

For this construction, Lemma \ref{auxiliar} is also true, which implies the following corollary, the proof of which is similar to Corollary \ref{partition}.

\begin{corollary}\label{partition2}
For every $i\geq 1$ and $\gamma\in \Gamma_i$, there exists $\alpha\in \{1,\dots,m\}$ such that 
$$
\eta(g)=\alpha \mbox{ for every } g\in \gamma J(i)R.
$$
\end{corollary}

Recall that $\{\sigma^{v^{-1}}C_n: v\in D_nR\}$ is a clopen partition of $\overline{O_\sigma(\eta)}$, where  
    $$
C_n=\{ x\in \overline{O_\sigma(\eta)}: \Per(x,\Gamma_n, \alpha)=\Per(\eta, \Gamma_n, \alpha), \mbox{ for every } \alpha\in \Sigma\}.
$$
  
For every $1\leq i\leq m$, let 
  $C_{n,i}=\{x\in C_n: x(g)=i \mbox{ for every } g\in J(n)R\}.$
Corollary \ref{partition2} implies that $\{C_{n,i}: 1\leq i\leq m\}$ is a covering of $C_n$.
Therefore, $$\cP_n=\{\sigma^{v^{-1}}C_{n,i}: 1\leq i\leq m, v\in D_nR\}$$ is a clopen partition of $\overline{O_\sigma(\eta)}$.

  \begin{lemma}\label{rel-partition}
 For every $n\geq 1$ and $1\leq j\leq m$, we have
  \begin{enumerate}
  \item $C_{n+1}\subseteq C_{n,\alpha_{n+1}}$, and
  \item $\sigma^{\gamma^{-1}}C_{n+1,j}\subseteq C_{n,j}$, for every $\gamma\in (\Gamma_{n}\cap D_{n+1})\setminus \{1_G\}$. 
  \end{enumerate}
  \end{lemma}
  \begin{proof}
  Since $\Per(\eta, \Gamma_{n})\subseteq \Per(\eta, \Gamma_{n+1})$, we have $C_{n+1}\subseteq C_n$. Furthermore, $D_nR\subseteq \Per(\eta, \Gamma_{n+1})$, which implies that
  $x(D_nR)=\eta(D_nR)$, for every $x\in C_{n+1}$.
  In particular, $x(g)=\eta(g)=\alpha_{n+1}$, for every $g\in J(n)R$.
  From this we get that $C_{n+1}\subseteq C_{n,\alpha_{n+1}}$.  
  
  \medskip
  
Using that $C_{n+1,j}\subseteq C_n$, we get that for every $\gamma\in \Gamma_{n}$,  $\sigma^{\gamma^{-1}}C_{n+1,j}\subseteq C_n$. On the other hand, if $\gamma\in (\Gamma_n\cap D_{n+1})\setminus \{1_G\}$ and $y\in C_{n+1,j}$, then Corollary \ref{partition2} implies that $\sigma^{\gamma^{-1}}(y)(g)=y(\gamma g)=j$, for every $g\in J(n)R$. This shows that $\sigma^{\gamma^{-1}}C_{n+1,j}\subseteq C_{n,j}$. 
  \end{proof}

 Let $\triangle$ be the convex set generated by the set of vectors $\{\vec{t}_1, \dots, \vec{t}_m\}\subseteq  \mathbb{R}^{m+1}$, where
 \begin{align*}
     \vec{t}_i=(t_1,\dots, t_{i-1},1-d+t_i,t_{i+1},\dots, t_m, t_\beta),
 \end{align*}
 with
 \begin{align*}
     t_\alpha=\lim_{n\to\infty}\dfrac{|D_nR\cap\Per(\eta,\Gamma_n,\alpha)|}{|D_nR|}, \mbox{ for each }\alpha\in \Sigma.
 \end{align*}
 
 That is,
 $$
 \triangle=\left\{\sum_{i=1}^m\lambda_i\vec{t}_i: \sum_{i=1}^m\lambda_i=1, \lambda_1,\dots, \lambda_m\geq 0\right\}.
 $$
 Since the vectors $\vec{t}_1,\dots, \vec{t}_m$ are linearly independent, the convex set $\triangle$ is a simplex.

\begin{proposition}\label{at-least}
The map \begin{align*}
p:M_G(\overline{O_\sigma(\eta)})\to \triangle,\; p(\mu)=(\mu([1]),\dots,\mu([m]),\mu([\beta])),
\end{align*} is affine, surjective and $p(\nu_i)=\vec{t_i}$ for each $1\leq i\leq m$. 
\end{proposition}

\begin{proof}
 For every $n\in\mathbb{N}$ and $\alpha\in \Sigma$, set $a_{n,\alpha} =|D_nR\cap \Per(\eta, \Gamma_n, \alpha)|$.
 
Let $C_{0,\alpha}=[\alpha]\cap \overline{O_\sigma(\eta)}$.
Note that for every $\alpha\in \{1,\dots, m\}$
    \begin{align*}
        C_{0,\alpha}=\bigcup_{g\in J(n)R}\sigma^{g^{-1}} C_{n,\alpha}\cup \bigcup_{g\in \Per(\eta,\Gamma_n,\alpha)\cap D_nR}\sigma^{g^{-1}}C_n.
    \end{align*}
    Therefore, for every invariant probability measure $\mu\in M_G(\overline{O_\sigma(\eta)})$ we have
    \begin{eqnarray*}
        \mu(C_{0,\alpha}) & = & |J(n)R|\mu(C_{n,\alpha})+a_{n,\alpha} \mu(C_n)\\
          &=& (|J(n)R|+a_{n,\alpha})\mu(C_{n,\alpha})+a_{n,\alpha}\sum_{j\in\{1,\dots,m\}\setminus\{\alpha\}}\mu(C_{n,j})\\
           &=& \frac{|J(n)R|+a_{n,\alpha}}{|D_nR|}\mu(C_{n,\alpha})|D_nR|+\frac{a_{n,\alpha}}{|D_nR|}\sum_{\substack{j\in\{1,\dots,m\}\setminus\{\alpha\}}}\mu(C_{n,j})|D_nR|.\\
    \end{eqnarray*}
    Taking a subsequence $(n_k)_{k\in\mathbb{N}}$ so that $\lim_{k\to \infty}\mu(C_{n_k,j})|D_{n_k}R|=\lambda_j\in [0,1]$, $j\in\{1,\dots,m\}$, we get 
    $$
    \mu(C_{0,\alpha})=\sum_{j=1}^m\vec{t}_j(\alpha)\lambda_j, \mbox{ with }\sum_{j=1}^m\lambda_j=1.
    $$
    On the other hand,
    \begin{align*}
        C_{0,\beta}=\bigcup_{g\in \Per(\eta,\Gamma_n,\beta)\cap D_nR}\sigma^{g^{-1}}C_n.
    \end{align*}
    Thus,
    \begin{align*}
        \mu(C_{0,\beta})=\dfrac{1}{|D_1|}\left(1-\dfrac{1}{|R|}\right)=\sum_{j=1}^m \vec{t_j}(\beta)\lambda_j.
    \end{align*}
    This implies that $p(\mu)=(\mu(C_{0,1}),\dots, \mu(C_{0,m}),\mu(C_{0,\beta}))$ belongs to $\triangle$. 
    Moreover, we have that $p(\nu_i)=\vec{t_i}$ for each $i\in\{1,\dots, m\}$.
    That $p$ is affine and surjective follows by its definition. 
\end{proof}

\begin{lemma}\label{partitions-determine-measures}
Let $p$ be  the map introduced in Proposition \ref{at-least}. 
If $\mu,\nu\in M_G(\overline{O_\sigma(\eta)})$ satisfy $p(\mu)=p(\nu)$, then $\mu|_{\mathcal{P}_n}= \nu|_{\mathcal{P}_n}$, for every $n\in \mathbb{N}$.
\end{lemma}
  
  \begin{proof}

For any $n\geq 1$,  Lemma \ref{rel-partition} implies that for every $1\leq i\leq m$  we have
   $$
  C_{n,i}=\left\{\begin{array}{ll}
                  \bigcup_{\gamma\in (D_{n+1}\cap \Gamma_n)\setminus \{1_G\}}\sigma^{\gamma^{-1}}C_{n+1,i} & \mbox{ if } i\neq \alpha_{n+1}\\
                  \bigcup_{\gamma\in (D_{n+1}\cap \Gamma_n)\setminus \{1_G\}}\sigma^{\gamma^{-1}}C_{n+1,\alpha_{n+1}}  \cup \bigcup_{j=1}^mC_{n+1,j} & \mbox{ if } i=\alpha_{n+1}.\\
                  \end{array}\right.
  $$
Thus if $\mu\in M_G(\overline{O_\sigma(\eta)})$, then
  $$
  \mu(C_{n,i})=\left\{\begin{array}{ll}
                  \left(\frac{|D_{n+1}|}{|D_n|}-1\right)  \mu(C_{n+1,i}) & \mbox{ if } i\neq \alpha_{n+1}\\
                  \left(\frac{|D_{n+1}|}{|D_n|}-1\right)  \mu(C_{n+1,\alpha_{n+1}})  + \sum_{j=1}^m\mu(C_{n+1,j}) & \mbox{ if } i=\alpha_{n+1}.\\
                  \end{array}\right.
  $$
Therefore, we have $A_n\mu^{(n+1)}=\mu^{(n)}$, where $\mu^{(n)}=(\mu(C_{n,1}),\dots, \mu(C_{n,m}))$ and 
 $A_n$ is the $m\times m$ integer matrix given by
  $$
  A_n(i,j)=\left\{ \begin{array}{ll}
                   \frac{|D_{n+1}|}{|D_n|}-1 &\mbox{ if } i=j\neq \alpha_{n+1}\\
                     \frac{|D_{n+1}|}{|D_n|} &\mbox{ if } i=j= \alpha_{n+1}\\
                   0 & \mbox{ if } i\neq j \mbox{ and } i\neq \alpha_{n+1}\\
                   1 & \mbox{ if } i\neq j \mbox{ and } i=\alpha_{n+1}.
     \end{array}\right.
  $$
   Since the matrices $A_n$ have linearly independent columns, they are invertible.
   Hence,
  $$ \mu^{(n+1)}=A_n^{-1}\dots A_1^{-1}\mu^{(1)}.$$
To connect $\mu^{(1)}$ and $p(\mu)$ we consider the $m+1\times m$ matrix $A_0$ given by
\begin{align*}
    A_0(i,j)=\begin{cases}
        1+|J(1)R|& \mbox{if }i=j=1\\
        |J(1)R|& \mbox{if }i=j\mbox{ and }2\leq i\leq m\\
        1&\mbox{if }i=1\mbox{ and }j\neq 1\\
        |R|-1& \mbox{if }i=m+1\\
        0& \mbox{otherwise.}
    \end{cases}
\end{align*}
From Lemma \ref{rel-partition} and the definition of the $C_{1,j}$'s we obtain 
\begin{align*}
    C_{0,i}=\begin{cases}
     \bigcup_{g\in J(1)R}\sigma^{g^{-1}}C_{1,1}\cup \bigcup_{j=1}^mC_{1,j}&\mbox{ if }i=1\\
     \bigcup_{g\in J(1)R}\sigma^{g^{-1}}C_{1,i}&\mbox{ if }2\leq i\leq m\\
     \bigcup_{g\in R\setminus\{1_G\}}\sigma^{g^{-1}}C_{1}&\mbox{ if }i=\beta.
    \end{cases}
\end{align*}
Therefore, 

\begin{align*}
    \mu(C_{0,i})=\begin{cases}
     (1+|J(1)R|)\mu(C_{1,1})+\sum_{j=2}^m \mu(C_{1,j})&\mbox{ if }i=1\\
     |J(1)R|\mu(C_{1,i})&\mbox{ if }2\leq i\leq m\\
\sum_{j=1}^m(|R|-1)\mu(C_{1,j})&\mbox{ if }i=\beta.
    \end{cases}
\end{align*}
  Thus, considering the $m+1$ coordinate for $i=\beta$, we obtain $p(\mu)=A_0\mu^{(1)}$. 
  As the rank of the matrix $A_0$ is $m$, there exists an $m\times m+1$ matrix $A_L$ such that $A_LA_0$ is the $m\times m$ identity map.
  Consequently, $\mu^{(1)}=A_Lp(\mu)$ and this concludes the proof of the lemma.
  \end{proof}
  Inspired by \cite[Remark 4.16]{CeCoGo23} and \cite[Proposition 18]{CeCo19}, we have the following proposition.
  \begin{proposition}\label{prop:injective}
      The map $p:M_G(\overline{O_\sigma(\eta)})\to \triangle$ introduced in Proposition \ref{at-least} is injective.
  \end{proposition}
  \begin{proof}
   Consider the set \begin{align*}
    \partial \overline{O_\sigma(\eta)}=\bigcup_{g\in G}\bigcap_{n\in\mathbb{N}}\bigcup_{i=1}^m\bigcup_{v\in D_nR\setminus D_nRg}\sigma^v C_{n,i}.
\end{align*}
We claim that for every $\mu\in M_G(\overline{O_\sigma(\eta)})$, $\mu(\partial \overline{O_\sigma(\eta)})=0$.
Indeed,
\begin{align*}
    \mu\left(\bigcup_{i=1}^m\bigcup_{v\in D_nR\setminus D_nRg}\sigma^{v^{-1}} C_{n,i}\right)=|D_nR\setminus D_nRg|\sum_{i=1}^m\mu(C_{n,i})=\dfrac{|D_nR\setminus D_nRg|}{|D_nR|}
\end{align*}
Thus,
\begin{align*}
    \mu\left(\bigcap_{n\in\mathbb{N}}\bigcup_{i=1}^m\bigcup_{v\in D_nR\setminus D_nRg}\sigma^{v^{-1}} C_{n,i}\right)\leq \lim \dfrac{|D_nR\setminus D_nRg|}{|D_nR|}=0.
\end{align*}
Consequently, $\mu(\partial \overline{O_\sigma(\eta)})=0$.
It means that the measures in $\overline{O_\sigma(\eta)}$ are determinated by their values in the set of the partitions $\mathcal{P}_n$, $n\in\mathbb{N}$.

On the other hand, let $x,y\in \overline{O_\sigma(\eta)}$ be two different elements such that $x,y\in \sigma^{v_n^{-1}}C_{n,i}$ for every $n\in\mathbb{N}$ and $v_n\in D_nR$.
The latter implies that $x(v_n^{-1}D_nR)=y(v_n^{-1}D_nR)$ for every $n\in\mathbb{N}$.
Moreover, as they are different elements, there exists $g\notin \bigcup_{n\in\mathbb{N}}v_n^{-1} D_nR$ with $x(g)\neq y(g)$.
Thus, we obtain that $v_n\in D_nR\setminus D_nRg^{-1}$ for each $n\in\mathbb{N}$ and consequently,
 $x,y\in\partial \overline{O_\sigma(\eta)}$.
It means that the elements in $\overline{O_\sigma(\eta)}\setminus \partial \overline{O_\sigma(\eta)}$ are separated by the atoms generated by $(\mathcal{P}_n)_{n\in\mathbb{N}}$. 
Consequently, every open set in $\overline{O_\sigma(\eta)}$ is a countable disjoint union of sets in the partitions $(\mathcal{P}_n)_{n\in\mathbb{N}}$ and a set in $\partial \overline{O_\sigma(\eta)}$. 
This concludes the proof.
  \end{proof}
  The following corollary is a direct consequence of propositions \ref{at-least} and \ref{prop:injective}.
  \begin{corollary}\label{Coro:m ergodic}
      The dynamical system $\overline{O_\sigma(\eta)}$ has exactly $m$ ergodic measures given by $\nu_i$, $1\leq i\leq m$.
  \end{corollary}

 Next, we guarantee that these measures have zero entropy. 
  First, for each $n\in\mathbb{N}$ and $i\in\{1,\dots, m\}$, define
  \begin{align*}
      E_n=\{x\in \Sigma^G: \Per(x,\Gamma_n,\alpha)=\Per(\eta,\Gamma_n,\alpha), \mbox{ for all }\alpha\in \Sigma\},
  \end{align*}
  and 
  \begin{align*}
      E_{n,i}=\{x\in E_n: x(g)=i, g\in J(n)R\}.
  \end{align*}
Note that $E_n\cap \overline{O_\sigma(\eta)}=C_n$ and $E_{n,i}\cap\overline{O_\sigma(\eta)}=C_{n,i}$.
Moreover, we have $\sigma^\gamma E_n=E_n$ for every $\gamma\in \Gamma_n$ as shown in \cite[Lemma 7]{CoPe08}.
This, in turn, implies that $E_n$ and $E_{n,i}$ are closed sets in $\Sigma^G$.
\begin{lemma}\label{Lemma:clopen}
    For each $n\in\mathbb{N}$, there exists a clopen subset $V_n$ in $\Sigma^G$ such that $E_n\subseteq V_n$ and $V_n\cap \overline{O_\sigma(\eta)}=C_n$. 
    Furthermore, this implies that the same is true if we replace $E_n$ by $E_{n,i}$, for all $i\in\{1,\dots, m\}$.
\end{lemma}
\begin{proof}
    As $E_n$ is closed, the set $A_n:=\Sigma^G\setminus E_n=\bigcup_{i\in I} W_i $ is open, where $I$ is a countable set and $W_i\subseteq \Sigma^G$ is a clopen set, $i\in I$.
    Furthermore, 
    $A_n\cap \overline{O_\sigma(\eta)} =\overline{O_\sigma(\eta)}\setminus C_n$.
    Since $\overline{O_\sigma(\eta)}\setminus C_n$ is a compact set in $\overline{O_\sigma(\eta)}$, there exists a finite set $I'\subseteq I$ such that $\left(\bigcup_{i\in I'}W_i\right)\cap \overline{O_\sigma(\eta)}=\overline{O_\sigma(\eta)}\setminus C_n$.
    The set $V_n:=\Sigma^G\setminus\left(\bigcup_{i\in I'} W_i\right)$ is the clopen subset in $\Sigma^G$ with the required properties.

    The second part is a direct consequence of the first part, as $E_{n,i}$ is the intersection of $E_n$ with a clopen set in $\Sigma^G$.
\end{proof}

Following \cite{CeCoGo23}, for each $k\in\mathbb{N}$ and $1\leq i\leq m$ we define the set $Z_{i,k}$ as follows.
\begin{align*}
    Z_{i,k}=\bigcup_{v\in D_{i+km-1}R}\sigma^{v^{-1}}C_{i+km-1,i}.
\end{align*}

\begin{lemma}\label{Lemma: Per_m}
    Let $n,m\in\mathbb{N}$ with $n<m$. 
    It holds that 
    \begin{align*}
        \Per(\eta_m,\Gamma_n,\alpha)=
        \Per(\eta,\Gamma_n,\alpha),\mbox{ for all }\alpha\in \Sigma.
    \end{align*}
    In other words, $\eta_m\in E_n$.
\end{lemma}
\begin{proof}
     Let $\alpha\in \Sigma$ and $dr\in \Per(\eta, \Gamma_{n},\alpha)$, with $r\in R$ and  $d\in D_{n}$.
     For every $\gamma\in \Gamma_n$, there exist $\gamma'\in \Gamma_{m}$ and $\tilde{\gamma}\in \Gamma_{n}\cap D_{m}$ such that $\gamma =\gamma'\tilde{\gamma}$.
Using that $\tilde{\gamma}d\in D_{m}$, we obtain
$ \eta_{m}(\gamma dr)= \eta_{m}(\gamma'\tilde{\gamma} dr)=\eta(\tilde{\gamma} dr)=\alpha.$
This implies that $\Per(\eta,\Gamma_{n},\alpha)\subseteq \Per(\eta_{m},\Gamma_{n},\alpha)$.

Now, assume there exists $dr\in J(k)R\cap \Per(\eta_m,\Gamma_n,\alpha)$ for some $n< k$, $r\in R$ and $\alpha\in \Sigma$.
Note that $k$ can be taken less than $m$.
Otherwise,  Lemma \ref{Lemma:Per} implies that $d=\tilde{\gamma} \tilde{d}$ where $\tilde{\gamma}\in (\Gamma_{m}\cap D_k)\setminus \{1_G\}$, and $\tilde{d}\in J(m)$.
Assume then $n<k\leq m$.

As $dr\in D_mR$, we have $\eta_m(dr)=\eta(dr)=\alpha=\alpha_{k+1}$.
By Lemma \ref{Lemma:Per}, there exist $\gamma \in (\Gamma_{k-1}\cap D_k)\setminus \{1_G\} $ and $d'\in J(k-1)$ such that $d=\gamma d'$.
Thus, $\eta_m(dr)=\eta_m(\gamma d' r)=\eta_m(d'r)=\eta(d'r)=\alpha_{k}$, which is a contradiction since $\alpha_k\neq \alpha_{k+1}$.
If $dr\in J(n)R\cap \Per(\eta_m,\Gamma_n,\alpha)$, then 
$\eta_m(dr)=\eta(dr)=\alpha=\alpha_{n+1}$.
For every $\gamma\in (\Gamma_n\cap D_{n+1})\setminus\{1_G\}$ it holds that $\gamma dr\in J(n+1)R$ by using Lemma \ref{Lemma:Per}.
Therefore, as $n+1\leq m$, we conclude
$\alpha=\eta_m(\gamma dr)=\eta(\gamma dr)=\alpha_{n+2}$ and again, we obtain a contradiction.

Thus, $\Per(\eta_m,\Gamma_n,\alpha)=\Per(\eta,\Gamma_n,\alpha)$, and we conclude the proof.
\end{proof}
  \begin{proposition}\label{Prop: Z's}
      For each $i\in\{1,\dots,m\}$ and $k\in\mathbb{N}$, it holds that $\nu_i(Z_{i,k})=1$.
  \end{proposition}
  \begin{proof}
  Let $s,k\in \mathbb{N}$ such that $s> k$.
Lemma \ref{Lemma: Per_m} implies $\eta_{i+sm-1}\in E_{i+km-1}$.
As $wJ(i+km-1)\subseteq J(i+sm-1)$ for every $w\in (\Gamma_{i+km-1}\cap D_{i+sm-1})\setminus \{1_G\}$ by Lemma \ref{Lemma:J's}, then the definition of $\eta$ implies that $\sigma^{v^{-1}}\eta|_{J(i+km-1)R}\equiv i$ for every $v\in \Gamma_{i+km-1}\cap D_{i+sm-1}$. 
Consequently, $\sigma^{v^{-1}}\eta_{i+sm-1}\in E_{i+km,i}$ for every $v\in \Gamma_{i+km-1}\cap D_{i+sm-1}\subseteq G'$ since $G'$ is abelian.
Hence,
\begin{align*}
    \dfrac{1}{|D_{i+km-1}R|}=\dfrac{|\Gamma_{i+km-1}\cap D_{i+sm-1}|}{|D_{i+sm-1}R|}\leq \eta_{i+sm-1}(E_{i+km-1,i}).
\end{align*}
By applying the Portmanteau Theorem, Lemma \ref{Lemma:clopen} and the fact that the sequence  $(\mu_{i+tm-1})_{t\in\mathbb{N}}$ converges weakly to $\nu_i$, we conclude 
\begin{align*}
\limsup_{s\to\infty}  \eta_{i+sm-1}(E_{i+km-1,i})\leq\lim_{s\to\infty}  \eta_{i+sm-1}(V_{i+km-1,i})=\nu_i(C_{i+km-1,i}),
\end{align*}
where $V_{i+km-1,i}$ are the clopen sets given by Lemma \ref{Lemma:clopen}.
This concludes the proof since $Z_{i,k}$ is a disjoint union of translations of the set $C_{i+km-1,i}$.
\end{proof}

The following corollary is a direct consequence of Proposition \ref{Prop: Z's}
\begin{corollary}\label{Coro: A}
    For every $1\leq i\leq m$, it holds that $\nu_i(A_i)=1$, where 
    \begin{align*}
        A_i=\bigcap_{g\in G}\bigcup_{k\in\mathbb{N}}\sigma^g Z_{i,k}.
    \end{align*}
\end{corollary}
We will use the set $A_i$ defined in the previous corollary to guarantee that $\pi|_{A_i}$ is a measure conjugacy between $(\overline{O_\sigma(\eta)},\sigma, \nu_i)$ and $(\overleftarrow{G}, \varphi,\mu)$, where $\mu$ is the unique invariant measure for the $G$-odometer $(\overleftarrow{G}, \varphi)$ associated to $\eta$.
The following lemma follows the same structure as \cite[Lemma 5.7]{CeCoGo23}. 
We add the proof for completeness.
\begin{lemma}\label{Lemma:injectivei}
    Let $1\leq i\leq m$ and $A_i$ be defined in Corollary \ref{Coro: A}.
    Then, the map $\pi|_{A_i}:A_i\to \pi(A_i)$ is injective.
\end{lemma}
\begin{proof}
Let $x,y\in A_i$ be such that $\pi(x)=\pi(y)$.
Let $g\in G$.
As $\pi(x)=\pi(y)$, there exists $w\in D_{i+k_gm-1}R$ so that $x,y\in \sigma^{w^{-1}} C_{i+k_gm-1}$ and hence, $\sigma^{g^{-1}}x,\sigma^{g^{-1}}y\in \sigma^{(wg)^{-1}} C_{i+k_gm-1}$.
If $v\in D_{i+k_gm-1}R$ satisfies $wg\in \Gamma_{i+k_gm-1}v$, then $\sigma^{g^{-1}}x,\sigma^{g^{-1}}y\in \sigma^{v^{-1}}C_{i+k_gm-1}$.
On the other hand,
since $x,y\in A_i$, there exists $k_g\in\mathbb{N}$ such that $x,y\in \sigma^g Z_{i,k_g}$.
Hence, $\sigma^{g^{-1}}x,\sigma^{g^{-1}}y\in \sigma^{v^{-1}}C_{i+k_gm-1,i}$.
Therefore, $\sigma^{g^{-1}}x=\sigma^{v^{-1}} x'$ and $\sigma^{g^{-1}}y=\sigma^{v^{-1}} y'$ for some $x', y'\in C_{i+k_gm-1,i}$.
Thus, we have that $x'(s)=y'(s)$ for every $s\in D_{i+k_gm-1}$ and in particular, $x'(v)=y'(v)$.
As $x(g)=x'(v)$ and $y(g)=y'(v)$, we conclude that $x(g)=y(g)$.
Since this is true for every $g\in G$, we deduce that $x=y$ and that $\pi|_{A_i}$ is injective.
\end{proof}
Inspired by \cite[Proposition 5.8]{CeCoGo23}, we have the following.
\begin{proposition}
    For every $i\in\{1,\dots, m\}$, the p.m.p systems $(\overline{O_\sigma(\eta)}, \sigma,G,\nu_i)$ and $(\overleftarrow{G},\varphi,G, \mu)$, with $\mu$ being the unique measure of $(\overleftarrow{G},\varphi,G)$, are measure conjugate.
\end{proposition}
\begin{proof}
Let $i\in\{1,\dots, m\}$.
Lemma \ref{Lemma:injectivei} guarantees that $\pi|_{A_i}:A_i\to\pi(A_i)$ is bijective.
Since $A_i$ is an invariant Borel set we deduce that $\pi(A_i)$ is a Borel set and $\pi|_{A_i}$ is a measurable map with measurable inverse which is $G$-equivariant (see, for instance, \cite[Theorem 2.8]{Gl03}).
Since $\overleftarrow{G}$ is uniquely ergodic, we obtain $\nu_i(\pi|_{A_i}^{-1}(B))=\mu(B)$ for every Borel set $B\subseteq \overline{O_\sigma(\eta)}$.
This finishes the proof since by Corollary \ref{Coro: A} we have that $\nu_i(A_i)=\mu(\pi(A_i))=1$.
\end{proof}
Using the variational principle and the previous proposition we obtain the following.
\begin{corollary}\label{Coro: 0entropy}
    For every $i\in\{1,\dots, m\}$, the p.m.p system $(\overline{O_\sigma(\eta)},\sigma, \nu_i)$ has zero entropy. 
    This implies that the dynamical system $(\overline{O_\sigma(\eta)}, \sigma,G)$ has zero topological entropy.
\end{corollary}
\subsubsection{Counting independence tuples for $\overline{O(\eta)}$}
For the last part, we count the fibers of the map $\pi:\overline{O_\sigma(\eta)}\to\overleftarrow{G}$ as in the previous subsection. 
The following lemma is shown in the same manner as Lemma \ref{aper-const} by using the sequence $(D_iR)_{i\in\mathbb{N}}$.

\begin{lemma}\label{Lemma:const}
Let $x\in \overline{O_\sigma(\eta)}$ and $\pi(x)=(\Gamma_it_i(x))_{i\in\mathbb{N}}$, where $t_i(x)\in D_iR$ for each $i\in\mathbb{N}$. 
It holds that  the map $x|_{t_i(x)^{-1}\gamma J(i)R}$ is constant for each  $i\in\mathbb{N}$ and  $\gamma\in \Gamma_i$, i.e., there exists $\alpha\in\{1,\dots, m\}$ such that $x(d)=\alpha$ for every  $d\in t_i(x)^{-1}\gamma J(i)R$.
In particular,  $x(d)=\alpha$ for every $d\in (t_i(x)^{-1}\gamma D_iR)\cap \Aper(x)$.
\end{lemma}
\begin{remark}
    {\rm For $x\in \overline{O_\sigma(\eta)}$, denote $\pi(x)=(\Gamma_it_i(x))_{i\in\mathbb{N}}$, where $t_i(x)\in D_iR$ for each $i\in\mathbb{N}$.
    Let $i\in\mathbb{N}$.
    In a similar way as in the previous subsection, 
for each $\zeta\in\Gamma_i$, we denote by $T_\zeta(x)$ the union of sets of the form $t_{i+j}(x)^{-1}\gamma_{i+j}^{\zeta}D_{i+j}R$, where for each $j\geq 1$, $\gamma_{i+j}^{\zeta}$ is the unique element in $\Gamma_{i+j}$ such that $t_{i+j}(x)^{-1}\gamma_{i+j}^\zeta D_{i+j}R\supseteq t_{i+(j-1)}(x)^{-1}\gamma_{i+(j-1)}^\zeta D_{i+(j-1)}R$ and $\gamma_{i}^\zeta=\zeta$, i.e., 
\begin{align*}
T_\zeta(x)=\bigcup_{j\in\mathbb{N}}t_{i+j}(x)^{-1}\gamma_{i+j}^\zeta D_{i+j}R.
\end{align*}
For $\overline{g}\in\overleftarrow{G}$ and $x,y\in\pi^{-1}(\overline{g})$ we observe that $t_{i}(x)^{-1}\zeta D_iR=t_i(y)^{-1}\zeta D_iR$ for every $i\in\mathbb{N}$ and $\zeta\in \Gamma_i$ since $t_i(x)=t_i(y)$.
Hence, $T_\zeta(x)=T_\zeta(y)$.}
\end{remark}

\begin{lemma}\label{decomGG}
   For each $x\in \overline{O_\sigma(\eta)}$, there exist $\alpha\leq 2^r[G:G']$ and $\zeta_{n_i}\in\Gamma_{n_i}$ for $1\leq i\leq \alpha$, such that 
    \begin{align*}
        G=\bigsqcup_{i=1}^{\alpha} T_{\zeta_{n_i}}(x).
    \end{align*}
\end{lemma}
\begin{proof}
For each $n\in\mathbb{N}$, consider $t_n(x)=d_n(x)r(x)$ for $d_n(x)\in D_n$ and $r(x)\in R$ ($r(x)$ is the same for every $n\in\mathbb{N}$). 
Since $G=t_n(x)^{-1}\Gamma_n D_nR$, there exist finite elements $\gamma_1,\gamma_2,\dots,\gamma_s\in\Gamma_n$ such that $B(n,1_{G'})R\subseteq \bigcup_{i=1}^{s}t_n(x)^{-1}\gamma_i D_nR$ satisfying that for each $1\leq j\leq s$, there exists $r\in R$ with $t_n(x)^{-1}\gamma_i D_n R\cap B(n,1_G)r'\neq\emptyset$.
These are the only elements that satisfy the previous condition.
Thus, 
\begin{align*}
    G=\bigsqcup_{i=1}^{\alpha}T_{\zeta_{n_i}}(x), \mbox{ for some } \alpha\in\mathbb{N}\cup\{\infty\} \mbox{ and }\zeta_{n_i}\in\Gamma_{n_i}, 1\leq i\leq \alpha.
\end{align*}

We claim that  $\alpha \leq 2^r[G:G']$. 
Indeed, assume $\alpha>2^r[G:G']$.
Thus, for some $n_0\in\mathbb{N}$ there exist $\alpha'\in\mathbb{N}$, with $\alpha \geq \alpha'>2^r[G:G']$, and $\zeta_1,\zeta_2,\dots,\zeta_{\alpha'}\in\Gamma_{n_0}$ in such a way that the sets $T_{\zeta_i}(x), i\in\{1,2,\dots, \alpha'\}$, are pairwise disjoint, $B(n_0,1_{G'})R\subseteq \bigsqcup_{i=1}^{\alpha'}t_{n_0}(x)^{-1}\zeta_i D_{n_0}R$ and for each $i\in \{1,2,\dots, \alpha'\}$, there exists $r\in R$ with $t_{n_0}(x)^{-1}\zeta_i D_{n_0}R\cap B(n_0,1_{G'})r\neq \emptyset$.
Let $r\in R$ and $i_1,\dots,i_{\alpha_r}$ be all the different elements in $\{1,\dots, \alpha'\}$ such that $B(n_0,1_{G'})r\subseteq \bigsqcup_{j=1}^{\alpha_r}t_{n_0}(x)^{-1}\zeta_{i_j} D_{n_0}R$ with  $t_{n_0}(x)^{-1}\zeta_{i_j}D_{n_0}R\cap B(n_0,1_{G'})r\neq \emptyset$, $1\leq j\leq \alpha_r$.
We have that $\sum_{r\in R}\alpha_r\geq \alpha'$. 
If $\alpha_r\leq 2^r$ for every $r\in R$, then $\alpha'\leq 2^r[G:G']$.
Therefore, there should be $r\in R$ such that $\alpha_r>2^r$.
By our assumption,
there exist elements $d_1,\dots,d_{i_{\alpha_r}}\in D_{n_0}$ and $r'\in R$ such that $t_{n_0}(x)^{-1}\zeta_{i_j}d_jr'\in B(1_G, n_0)r$.
Consequently, $B(s,t_{n_0}(x)^{-1}\zeta_{i_j} d_jr')\subseteq B(n_0+s,1_{G'})r$ for every $s\in\mathbb{N}$.
Hence, for $s\in\mathbb{N}$ we have 
\begin{align*}
\bigsqcup_{j=1}^{\alpha_r}    B(s,t_{n_0}(x)^{-1}\zeta_{i_j} d_jr')\cap (t_{n_0+s}(x)^{-1}\gamma_{n_0+s}^{\zeta_{i_j}} D_{n_0+s}r')\subseteq B(n_0+s,1_G)r,
\end{align*}
which implies that 

\begin{align}\nonumber \label{alpha-prime}
b(n_0+s)\geq& \sum_{j=1}^{\alpha_r}   |B(s,t_{n_0}(x)^{-1}\zeta_{i_j} d_jr')\cap (t_{n_0+s}(x)^{-1}\gamma_{n_0+s}^{\zeta_{i_j}} D_{n_0+s}r')|\\
\nonumber\geq&\sum_{j=1}^{\alpha_r}   |B(s, d_j)\cap (\zeta_{i_j}^{-1}\gamma_{n_0}^s(x)\gamma_{n_0+s}^{\zeta_{i_j}} D_{n_0+s})|
\end{align}
Now, we conclude the proof using Lemma \ref{1/2r} as in Lemma \ref{decomG}.
\end{proof}

As in the previous subsection, Lemma \ref{Lemma:const} implies the following lemma.

\begin{lemma}\label{Lemma:Const-T}
Let $x\in \overline{O_\sigma(\eta)}$ be a non-Toeplitz element and $i\in\mathbb{N}$. 
For every $\zeta\in\Gamma_i$ we have that there exists $\alpha\in\{1,\dots, m\}$ such that $x(d)=\alpha$ for every  $d\in T_\zeta(x)\cap\Aper(x)$.
\end{lemma}

The proof of the following proposition is similar to that of Proposition \ref{Regio-prox-bound}.
\begin{proposition}\label{Regio-prox-bound2}
 For every $x\in \overline{O_\sigma(\eta)}$, it holds that $|\pi^{-1}(\{\pi(x)\})|\leq m^{2^r[G:G']}$. 
\end{proposition}

\begin{proof}[Proof of Theorem \ref{theo-main}] When the group $G$ is assumed to be isomorphic to $\mathbb{Z}^r$, for some $r\in\mathbb{N}$, the result follows from Corollary \ref{maintheo:Zr}.
Now, assuming that $[G:G']\geq 2$, the result follows  directly from Corollary \ref{Coro:m ergodic}, Corollary \ref{Coro: 0entropy}, Proposition \ref{Regio-prox-bound2}, Proposition \ref{htop*} and Proposition \ref{It-to-Reg}.
    
\end{proof}

\begin{remark}
    In section \ref{Abelianfinite} we have used a specific sequence $(\Gamma_i)_{i\in\mathbb{N}}$ of finite index subgroups of $G$ to take advantage of the structure of the sequence of fundamental domains given by $(D_i)_{i\in\mathbb{N}}$. 
    Nevertheless, the same can be proved if we use a sequence  $(\Gamma_i)_{i\in\mathbb{N}}$ where instead of taking the canonical generators of $\mathbb{Z}^r$ in (\ref{Gammas}) we consider a set of $r$ elements in $\mathbb{Z}^r$ which are $\mathbb{Z}$-linearly independent and with some modification in (\ref{rectangle}).
    Besides using the geometry of $\mathbb{Z}^r$, we suspect that an analogous procedure works if we suppose that the group $G$ is a torsion-free finitely generated group of sub-exponential growth.
\end{remark}
\section{Questions}\label{Sec-5}

\begin{question}
Are there versions of Theorem \ref {theo-main} and Proposition \ref{theo-main-2} for uniquely ergodic systems?
\end{question}

\begin{question}
Is  Theorem \ref{theo-main} still valid when  $G$ is assumed to be a essentially amenable group, a polycyclic group, a solvable group or a finitely generated group of sub-exponential growth?
\end{question}

A dynamical system $(X,\varphi, G)$ such that $\IT_2(X)\setminus \triangle^{(2)}(X)=\emptyset$ is called \emph{tame} (see \cite{KeLi07}). 
For an \textit{$n$-tame} system we mean a 
dynamical system that no contains $n$-IT-tuples with different elements. 
Theorem \ref{theo-main} guarantees that for a group $G$ containing a finite index normal subgroup that is isomorphic to $\mathbb{Z}^r$, for some $r\in \mathbb{N}$, 
it is always possible that for each $n\in\mathbb{N}$, to find $m\geq n$ such that there exists a $G$-dynamical system that is $m+1$-tame but not $m$-tame.

\begin{question}
Is there a group $G$ where all 3-tame dynamical systems over this group are tame?
\end{question}

In \cite{KeLi07}, a Toeplitz subshift $X$ over $\mathbb{Z}$ was constructed that is tame but nonnull.
In other words, this Toeplitz subshift satisfies that $\IT_2(X)\setminus \triangle^{(2)}(X)=\emptyset$ but $\IN_2(X)\setminus\triangle^{(2)}(X)\neq \emptyset$. Note that for the examples and groups treated in this document, we have that $\IT_n(X)\setminus \triangle^{(n)}(X)=\emptyset$ if and only if $\IN_n(X)\setminus \triangle^{(n)}(X)=\emptyset$.
This suggests the following question.

\begin{question}
Given a countable infinite group $G$ and $n\in\mathbb{N}$, are there examples of dynamical systems such that
$\IT_n(X)\setminus \triangle^{(n)}(X)=\emptyset$ and $\IN_n(X)\setminus \triangle^{(n)}(X)\neq\emptyset$?
\end{question}

\end{document}